\documentclass[12pt]{article}
\usepackage{geometry} 
\usepackage{amsmath}
\usepackage{amssymb}
\usepackage{mathabx}
\usepackage{amscd}
\usepackage{latexsym}
\usepackage{amsthm}
\usepackage{mathrsfs}
\usepackage{color}
\usepackage{amsfonts}
\usepackage[shortlabels]{enumitem}
\usepackage{array}
\usepackage{setspace}
\usepackage{mathtools}
\usepackage{tikz}
\usepackage{tikz-cd}
\usepackage{hyperref}
\usetikzlibrary{arrows}
\usetikzlibrary{knots}
\usepackage{ytableau}
\usepackage{enumitem}
\usepackage{subcaption}

\usepackage{stmaryrd,graphicx}

\newtheorem{thm}{Theorem}[section]
\newtheorem*{thm*}{Theorem}
\newtheorem{cor}[thm]{Corollary}
\newtheorem*{cor*}{Corollary}
\newtheorem{lem}[thm]{Lemma}
\newtheorem*{lem*}{Lemma}
\newtheorem{prop}[thm]{Proposition}
\newtheorem*{prop*}{Proposition}

\newtheorem*{thma}{Theorem A}
\newtheorem*{thmb}{Theorem B}

\theoremstyle{definition}
\newtheorem{defn}[thm]{Definition}
\newtheorem*{defn*}{Definition}

\newtheorem{conjecture}[thm]{Conjecture}
\newtheorem*{conjecture*}{Conjecture}
\newtheorem*{conja}{Conjecture A}
\newtheorem*{conja1}{Conjecture A$'$}

\newtheorem*{condition*}{Condition}
\newtheorem*{assumption*}{Assumption}

\theoremstyle{remark}
\newtheorem{rem}[thm]{Remark}
\newtheorem*{rem*}{Remark}
\newtheorem{example}[thm]{Example}

\newtheorem*{problem*}{Problem}

\DeclareMathOperator{\triv}{triv}
\DeclareMathOperator{\sgn}{sgn}
\DeclareMathOperator{\Ind}{Ind}

\DeclareMathOperator{\diag}{diag}

\DeclareMathOperator{\maj}{maj}

\newcommand{\pitoalpha}{\alpha}

\newcommand{\affsl}{Y}

\newcommand{\inneighbors}{N^-}
\newcommand{\outedges}{E^+}
\newcommand{\hessmap}{h}
\newcommand{\outneighbors}{N^+}
\newcommand{\hovar}{A}
\newcommand{\covar}{\xi}

\newcommand{\aaf}{\mathbb{A}_{af}}
\newcommand{\aafgl}{\widehat{\mathbb{A}}_{af}}

\newcommand{\awgsl}{W}
\newcommand{\awg}{W}

\newcommand{\uio}{D}

\DeclareMathOperator{\lt}{LT}

\newcommand{\csf}{\mathcal{X}}
\newcommand{\csfa}{\xi}

\newcommand{\bruleq}{\leq_{bru}}

\newcommand{\BA}{\mathbb A}

\newcommand{\CL}{\mathcal L}
\newcommand{\CV}{\mathcal V}

\newcommand{\CZ}{\mathcal Z}

\newcommand{\yk}{Y_k}
\newcommand{\yp}{Z}
\newcommand{\zk}{Z_k}
\newcommand{\ypk}{Z_k}
\newcommand{\trans}[1]{t(#1)}

\newcommand{\wmin}{-}

\newcommand{\wt}{\widetilde}

\newcommand{\eps}{\epsilon}

\newcommand{\aut}{\mathrm{aut}}

\newcommand{\frob}{\mathcal{F}}

\newcommand{\acts}{\righttoleftarrow}

\newcommand{\cO}{\mathcal{O}}

\newcommand{\rsh}{\mathcal{H}}

\newcommand{\labtocmp}{\alpha}

\newcommand{\aff}{Y}

\DeclareMathOperator{\revl}{rev}

\DeclareMathOperator{\paff}{aff}
\DeclareMathOperator{\shuff}{Std}

\DeclareMathOperator{\Gr}{Gr}

\DeclareMathOperator{\Fl}{Fl}

\DeclareMathOperator{\GL}{GL}

\DeclareMathOperator{\Tor}{Tor}

\DeclareMathOperator{\Lie}{Lie}

\DeclareMathOperator{\Par}{Par}

\DeclareMathOperator{\End}{End}

\DeclareMathOperator{\Sym}{\Lambda}

\DeclareMathOperator{\inv}{inv}
\DeclareMathOperator{\dinv}{dinv}

\newcommand{\wf}{\tau}
\newcommand{\Torc}[3]{\Tor_{#1}^{#2}(#3,\C)}

\newcommand{\dyckpath}{\pi}
\newcommand{\partialdyckpath}{\pi'}

\newcommand{\Ht}{\tilde{H}}

\newcommand{\nlist}{\mathbf}
\newcommand{\an}{{\nlist{a}}}
\newcommand{\bn}{{\nlist{b}}}

\newcommand{\mn}{{\nlist{m}}}
\newcommand{\sn}{{\nlist{s}}}
\newcommand{\tn}{{\nlist{t}}}
\newcommand{\xn}{{\nlist{x}}}
\newcommand{\yn}{{\nlist{y}}}
\newcommand{\zn}{{\nlist{z}}}

\newcommand{\un}{{\nlist{u}}}
\newcommand{\Z}{\mathbb{Z}}
\newcommand{\C}{\mathbb{C}}

\DeclareMathOperator{\rank}{rank}

\author{Erik Carlsson, Anton Mellit}

\begin{document}
	\title{GKM spaces, and the signed positivity of the 
		nabla operator}
	\maketitle
	
	\begin{abstract}
		We show that the Frobenius character of the equivariant Borel-Moore 
		homology $H_*^T(\ypk)$ of a certain positive $GL_n$-version of the unramified affine Springer fiber studied by Goreski, Kottwitz and MacPherson 
		\cite{goresky2004unramified} is computed by the matrix coefficients of the $\nabla^k$-operator, which acts diagonally in the modified Macdonald basis. 
		We do this by relating the combinatorial formula of \cite{carlsson2020combinatorial} for the $\nabla^k$-operator to the GKM paving of $Z_k$, 
		and we give an algebraic presentation of the above homology as an explicit
		submodule of the Kostant-Kumar nil Hecke algebra \cite{kostant1986nil}. 
		We then study a certain open locus $U_k \subset \ypk$, 
		and reduce a long-standing conjecture of \cite{bergeron1999identities},
		which predicts the sign of the coefficients 
		of the Schur expansion of $\nabla^k$,
		to a vanishing conjecture about the homology groups of $U_k$.
		The latter conjecture is in turn  
		reduced to a vanishing conjecture for certain open loci
		of the regular semisimple Hessenberg varieties
		which are indexed by partial Dyck paths. 
	\end{abstract}
	
	\section{Introduction}

	Let $\nabla$ be the nabla operator of \cite{bergeron1999identities}, which is
	diagonal in the basis of modified Macdonald polynomials.
	Let $X,Y$ be alphabets for two different sets of symmetric
	functions, and let $\nabla=\nabla_X$
	act in the $X$-variables.
	We adopt the usual plethystic notation of 
	\cite{Haiman01vanishingtheorems}, in which
	$f[A]$ is the result of substituting $p_k=A|_{v=v^k}$
	where $v$ ranges over all variables that appear in $A$.
	In \cite{carlsson2020combinatorial}, we proved
	\begin{thm}
		\label{thm:oldthm}
		For any $k\geq 1$, we have
		\begin{equation}\label{eq:oldnabla1}
			\nabla^k e_n\left[\frac{XY}{(1-q)(1-t)}\right] 
			= \sum_{[\mn, \an, \bn]}  \frac{t^{|\mn|} 
				q^{\dinv_k(\mn, \an,\bn)}}{(1-q)^n \aut_q(\mn,\an,\bn)}
			X_{\an} Y_{\bn}.
		\end{equation}
	\end{thm}
	Here $\mn \in \mathbb{Z}_{\geq 0}^n$,
	$\an,\bn\in\mathbb{Z}_{\geq 1}^n$ are labels for the variables
	$X_\an,Y_\bn$,
	and $[\mn,\an,\bn]$ denotes an orbit under the
	$S_n$ action by simultaneously reordering
	the elements. The absolute value $|\mn|$
	is just the sum of the elements,
	$\dinv_k(\mn,\an,\bn)$ is a certain combinatorial
	statistic extending the one by the same name 
	that appears in the Shuffle Theorem
	\cite{haglund2005combinatoriala,haglund2012compositional,carlsson2018proof}, and $\aut_q(\mn,\an,\bn)$ is an explicit
	product of $q$-factorials.

	The $\dinv$ statistic in equation \eqref{eq:oldnabla1}
	is essentially the dimension of certain
	Schubert cell in the
	affine Springer fiber $\yk$
	of \cite{goresky1997koszul,goresky2003purity,goresky2004unramified}, described in Section 
	\ref{sec:unramgkm}.
	The Frobenius character and the automorphism factors were 
	discovered by studying an explicit
	module $M_k$ over the polynomial ring
	$\C[x_1,...,x_n,y_1,...,y_n]$, which 
	was conjectured to be the image of $ P\otimes
	\mathcal{O}(k)$ under the Haiman-Bridgeland-King-Reid
	isomorphism \cite{bridgeland2007stability,Haiman01vanishingtheorems,haiman2001hilbert}. As we have recently learned, 
	this has been shown using
	some new results in \cite{alvarez2021affine} using 
	Cherednik algebras. 
	The results of \cite{carlsson2020combinatorial} did not attempt to connect
	with Haiman's theory, and instead verified
	equation \eqref{eq:oldnabla1} by establishing that 
	the expected sum satisfies
	the defining properties of the $\nabla^k$ operator, namely triangularity,
	symmetry, and leading term.
	An alternative proof was given by interpreting 
	the sum as an enumeration over parabolic vector bundles over $\mathbb{P}^1$, which had been related to Macdonald polynomials in \cite{mellit2020poincare}.
	
	\subsection{The affine Springer fiber} 
	We describe now the geometric picture. Let $G=GL_n$, with
	standard maximal torus $T\subset G$, which has Weyl group $S_n$.
	The $GL_n$ affine flag variety $Y$ is the quotient $G(\mathcal{K})/I$ where $\mathcal{K}=\C((z))$ and $I$ is the stabilizer of the standard infinite flag of lattices in $\mathcal{K}^n$. It has a structure of an ind-variety, 
	and its connected components are in bijection with integers, each one being isomorphic to its $SL_n$ version. 
	We have an action of the
	$(n+1)$-dimensional torus $\tilde{T} \supset T$ 
	which includes the loop rotation parameter, 
	and the fixed points are in bijection with the
	affine Weyl group $W$, which is the set of (extended) affine permutations.
	The Schubert decomposition $G(\mathcal{K}) = 
	\bigsqcup_{w\in W} I w I$ corresponds to 
	the paving of $Y$ by $I$-orbits $\Omega_w$, 
	which are isomorphic to affine spaces. The set of affine permutations 
	$W$ is endowed with Bruhat order so that any lower set of $W$ 
	corresponds to a closed subvariety of $Y$. The set of positive affine permutations $W_+ \subset W$,
	which take positive integers to positive integers,
	is an example of a lower set, and thus corresponds to a closed subvariety which we denote by $Z\subset Y$. Every connected component of $Z$ consists of a finite number of cells, and therefore is an ordinary variety.
	
	Let $\yk=Y_{\gamma_k}\subset Y$ denote the unramified affine
	Springer fiber studied in \cite{lusztig1991fixed,goresky2004unramified} in type A,
	associated to the topologically nilpotent element
	\[\gamma_k=\diag(a_1z^k,...,a_nz^k),\] 
	where $a_i\in \C^*$ are distinct, which is simply the set of points of $Y$ fixed by $1+\gamma_k$.
	It is shown in \cite{goresky2003purity} 
	that $\yk$ is paved by 
	affine spaces given by the intersections
	$\yk \cap \Omega_w$. The torus $T$
	as well $\tilde{T}$, act on $Y$ and preserve $\yk$, the cells $\Omega_w$, and $Y^+$, and so 
	they also acts on the intersection $\ypk=\yk \cap Z$, which we call \emph{the positive part}. 
	The $T$-equivariant Borel-Moore homology is denoted by $H^T_*$ and is a module over $S=H^*_T(pt)=\C[x_1,...,x_n]$, where the $x_i$ are the generators
	of $\Lie(T)$. We use the convention the $x_i$ are positively graded,
	forcing homological degree to be graded in the negative direction,
	so that the corresponding Frobenius character is a formal 
	Laurent series in $q$. 
	We have inclusions
	\begin{equation}\label{eq:embeddings}
		H^T_*(Y_k) \subset H^T_*(Y) \subset 
		F\otimes_S H^T_*(Y^T)\cong F\cdot W,
	\end{equation}
	where $F=\C(x_1,...,x_n)$ is the field of fractions of $S$,
	$F\cdot W=F\otimes \C[W]$. 
	These spaces are GKM
	with respect to a larger torus $\tilde{T} \supset T$ which
	includes the loop rotation parameter, which means that
	$H_*^{\tilde{T}}(\yk)$ may be described by
	explicit relations in $\tilde{F}\cdot W_+$.

	The non-affine Weyl group $S_n$ acts on 
	$F\cdot W$ on the left and on the right. The left action is called the \emph{dot action} and it intertwines the action of $S$, whereas the right action is called the \emph{star action} and it commutes with $S$. Both of these actions descend to $H^T_*(\yk)$. The dot action on $H^T_*(Y_k)$ becomes identified with the space level action of $S_n$ on $Y$ by the left multiplication. Multiplication by $w\in W$ sends $Y_{\gamma_k}$ to 
	$Y_{w(\gamma_k)}$, so we do not have a space level action per se, but the homology of $Y_k$ for different choices of $\gamma_k$ are all naturally identified via the parallel transport, or alternatively via the embedding 
	$H^T_*(Y_k) \subset H^T_*(Y)$.

	The star action on the other hand is the Springer action, 
	and can be understood as follows. For any parabolic $P_\alpha\supset B$ we have versions of the spaces $Y_k$, $Y$ where we use $I_\alpha$, which is the subgroup of $G(\mathcal{K})$ preserving the corresponding standard 
	partial flag of lattices of $\mathcal{K}^n$ instead of $I$, and is an example of a \emph{parahoric} subgroup. Denote these spaces by $Y_k^{\alpha}\subset G(K)/I_\alpha$. The natural projections
	\[Y_k\to Y_k^{\alpha},\quad Y\to Y^{\alpha}\]
	are proper and $T$-equivariant, and therefore induce pushforward maps on $H^T_*$. The Springer action is the unique action which identifies these pushforward maps with the projection maps to $S_\alpha$-invariants, where $S_\alpha=S_{\alpha_1}\times \cdots\times S_{\alpha_l}=S_n\cap P_\alpha$ is the corresponding Young subgroup.
	
	Both actions descend to $H_*^T(Z_k)=H_*^T(Y_k)\cap F\cdot W_+$.
	
	We may now consider the bigraded Frobenius character.
	We encode the Frobenius character of the star and dot action by 
	symmetric functions in the alphabet $X$ and $Y$ respectively.
	We have that $H^T_*(Z_k)$ is bigraded so that the first grading 
	is half of the cohomological grading, i.e. negative of the homological grading, which exists only in even degree,
	and is encoded by powers of the variable $q$. The second is 
	the space-level grading by the index of the connected 
	component of $Z_k$, which is nonempty only for nonnegative
	integers, encoded by powers of $t$.
	\begin{thma} We have
		\begin{enumerate}[(a)]
			\item \label{item:thmanilhecke}
			The equivariant homology of the Springer
			fiber is identified with an explicit submodule of
			the Kostant-Kumar nil-Hecke algebra
			\[
			H_*^T(Y_k)=\aafgl\cap  \Delta(\xn)^{-k}S \cdot W,
			\]
			where $\Delta(\xn)=\prod_{i<j} (x_i-x_j)$, and the 
			intersection is taken in $F\cdot W$.
			\item \label{item:thmafrob}
			The Frobenius character of the positive part is given by
			\begin{equation}\label{eq:new nabla}
				\frob_{Y,X} H_*^T(\ypk)=q^{-k\binom{n}2} \omega_X \nabla^k e_n\left[\frac{XY}{(1-q)(1-t)}\right].
			\end{equation}
		\end{enumerate}
	\end{thma}

	In part \ref{item:thmafrob},
	which we explain first, 
	$\omega_X$ denotes the involution $s_\lambda\to s_{\lambda'}$ applied with respect to the $X$ variables.
	This result is proved using Theorem A
	from \cite{carlsson2020combinatorial},
	by establishing that the right side
	of \eqref{eq:oldnabla1} computes the 
	Frobenius character.
	Without the dot action, the right hand side is just counting cells in the affine paving of the parabolic
	version of $Y_k$ using the Springer action, 
	with the number $k\binom{n}2-\dinv_k$ being the dimension of a cell. 
	To compute the dot action,
	we use the alternative presentation
	of \eqref{eq:oldnabla1} as a sum of certain LLT polynomials
	$\csfa_{\pi}[Y;q]$, which are the Frobenius characters of the regular semisimple Hessenberg varieties by a conjecture of Shareshian and Wachs, proved in 
	\cite{brosnan2018hessenberg,shareshian2016chromatic}. 
	We then verify that our decomposition into LLT polynomials precisely matches the paving of 
	$Y_k$ by affine bundles over Hessenberg
	varieties from \cite{goresky2003purity}, and restricting this paving to $Z_k$ we obtain the result. The automorphism
	factors appear because the fibers of the flag versions of
	these Hessenberg varieties over certain parabolic versions
	are products of usual flag varieties.

	In part \ref{item:thmanilhecke}, 
	we identify $H^T_*(Y_k)$ with an explicit bigraded $S_n\times S_n$-module which we now define.
	The group of affine permutations $W$ is generated by the group of finite permutations $S_n$ and the rotation element $\rho_i=i+1$.
	We have the action of Kostant and Kumar's nil Hecke algebra
	$\aaf$ on the $F$-vector space $F\cdot W$
	with basis vectors labeled by $p_w$ for $w\in W$,
	see \cite{kostant1986nil,lam2014schur}. The submodule
	\[\aafgl=\sum_{d\in \mathbb{Z}} 
	\aaf \rho^d \subset F\cdot W\]
	is free over $\aaf$, and is an algebraic presentation
	of the equivariant Borel-Moore homology of the
	affine flag variety $Y$ for $GL_n$ with the action
	of the small torus $T\subset GL_n$.
	The intersection in the first part 
	may now be understood as the $S$-submodule
	consisting of elements of $\aafgl$ whose denominator divides $\Delta(\xn)^k$.
	

	\subsection{An open subvariety and nabla positivity}

	Finally, we present a potential 
	application which can be seen as
	a categorification of the famous nabla positivity
	conjecture of \cite{bergeron1999identities},
	which states that the coefficients $c_{\lambda,\mu}(q,t)$
	of the Schur expansion
	\[(-1)^{\iota(\lambda')}
	\nabla^k s_{\lambda}=\sum_{\mu}c_{\lambda,\mu}(q,t) s_{\mu}\]
	are nonngegative, and $\iota(\lambda)$ is the number of boxes of $\lambda$ below the
	main diagonal (see Conjecture \ref{conj:nabpos}).

	In our approach, we consider the action of
	$\C[y_1,...,y_n]=\C[\mathbb{Z}_{\geq 0}^n]$ on $H_*^T(Z_k)$
	by the left (dot) multiplication by
	\[y_i\in \mathbb{Z}_{\geq 0}^n\subset W_+ \subset W,\]
	which sends $i\mapsto i+n$,
	and fixes the other elements of
	$\{1,...,n\}$. These are induced from the restriction of
	the space-level action of the translation group
	$\mathbb{Z}_{\geq 0} \subset \mathbb{Z}^n\subset W$ 
	on $\ypk$. We thus obtain
	that $H_*^T(Z_k)$ is a bigraded $S_n\times S_n$ equivariant
	module over $R=\C[x_1,...,x_n,y_1,...,y_n]$, in which
	the left action of $S_n$ intertwines the diagonal
	action on $R$.
	
	We then prove that the action of $\C[y_1,...,y_n]$ is free,
	which is analogous to Haiman's freeness of polygraph rings over one set of variables on the ``coherent side'' of the BKR ismorphism \cite{haiman2001hilbert,Haiman01vanishingtheorems},
	and we have learned would follow from the results of 
	\cite{alvarez2021affine}, once one has the
	identification of $M_k$ from Theorem A. We give two different proofs of the freeness. One proof is geometric: 
	we study a certain family of subvarieties
	$Z^i_k\subset \ypk$, which we prove are $S_n$-translates
	of unions of intersected Schubert cells, and
	satisfy $H_*^T(Z^i_k)=y_i H^T_*(\ypk)$ via the inclusion map.
	We then show that their homologies
	satisfy a ``distributive lattice'' type property in
	Proposition \ref{prop:dist-lat} and conclude that
	the $y_i$ define a regular sequence on $H_*^T(\ypk)$.
	
	Another proof uses explicit symmetry between $x$ and $y$ variables, which may be interesting on its own, see Theorem \ref{thm:gl2}. This symmetry is a consequence of a $\GL_2$ action constructed with the the help
	of explicit differential operators using the presentation of $H_*^T(Z_k)$ as the intersection $\aafgl\cap  \Delta(\xn)^{-k}S\cdot W_+$. It turns out our operators preserve both $\aafgl$ and $\Delta(\xn)^{-k}S\cdot W_+$.
	
	We then consider the quotient module
	$H_*^T(Z_k)/(y_1,...,y_n)H_*^T(Z_k)$, for which we also give a geometric interpretation:
	let $U_k=\ypk-Z$ for
	$Z=Z^1_k\cup \cdots \cup Z^n_k$ be the
	complement. Then we have
	the long exact sequence 
	\[\cdots \to H^T_*(Z) \to H^T_*(Z_k) \to H^T_*(U_k) \to H^T_{*-1}(Z) \to \cdots \]
	Using the distributive lattice property, we show that
	this splits into short exact sequences, so that
	$H^T_*(U_k)$ is supported in even 
	degrees, and we have that (Corollary \ref{cor:nkmod})
	\[
	H^T_*(U_k) \cong H_*^T(Z_k)/(y_1,...,y_n)H_*^T(Z_k).
	\]
	Note that this does not imply that $U_k$ is equivariantly formal,
	only that \emph{equivariant} Borel-Moore homology is supported
	in even degree. Indeed, equivariant formality 
	would imply that $H^T_*(U_k)$ is free
	over $\C[x_1,...,x_n]$ which is untrue, 
	and in fact $U_k$ has odd non-equivariant homology.
	
	Then the $i$th Tor group
	$\Tor_i^R(H^T_*(Z_k),\C)\cong \Tor_i^{S} (H^T_*(U_k),\C)$ inherits
	the left and right
	action of $S_n$ as well as the bigrading. 
	We propose the following categorification of nabla 
	positivity: 
	\begin{conja}
		The multiplicity of the irreducible
		representation $\chi_\lambda$
		of the left action of $S_n$ 
		on $\Tor_i^R(H^T_*(Z_k),\C)\cong \Tor_i^{S} (H^T_*(U_k),\C)$
		is nonzero for at most one value of
		$i$, namely $i=\iota(\lambda)$.
	\end{conja}
	The equivariant homology $H_*^T(U_k)$ is pure, and therefore by \cite{franz2005weights} the $\Tor$ groups above can be interpreted as steps in the weight filtration in the corresponding non-equivariant homology $H_*(U_k)$. For instance, $\Tor_0$ corresponds to the pure part of the homology, $\Tor_1$ to the part one degree off from pure, and so on, see Corollary \ref{cor:weight U_k}.
	
	We further refine Conjecture A
	by intersecting $U_k$ with the 
	the paving by affine bundles over Hessenberg varieties of \cite{goresky2003purity}.
	Interestingly, while the Hessenberg varieties from this paving are labeled by Dyck paths, their intersections
	with $U_k$ have slightly more structure, and in fact are
	determined by \emph{partial} Dyck paths, in which the number
	of missing steps $l$ corresponds to the number of nontrivial
	intersections with the $Z^i_k$. The homology groups of these varieties
	$H_*^T(U_{\pi,l})\cong H^*_T(U_{\pi,l})$, where
	the second isomorphism follows by smoothness of $U_{\pi,l}$,
	are further conjectured to satisfy
	the vanishing property in the following refinement of 
	Conjecture A:
	\begin{conja1}
		Conjecture A holds with $U_{\pi,l}$ in place of $U_k$.
	\end{conja1}
	Again, $H_*^T(U_{\pi,l})$ is pure, and therefore the conjecture can be interpreted in terms of the weight filtration on $H_*(U_{\pi,l})\cong H^*(U_{\pi,l})$, see Corollary \ref{cor:weight U pi}.
	
	We then prove our second result:
	\begin{thmb}
		Conjecture A implies the nabla positivity conjecture
		of \cite{bergeron1999identities}.
		Conjecture A$'$ implies Conjecture A.
	\end{thmb}
	Both conjectures A and A$'$ are 
	supported by algebraic calculations
	using Gr\"{o}bner bases in SAGE and MAPLE.
	Additionally, in Proposition \ref{prop:nabchi} we prove an explicit formula for the
	Frobenius character of $\chi_{\pi,l}=\frob H_*^T(U_{\pi,l})$
	as a quasisymmetric function, 
	which may be readily observed on a computer
	to satisfy the signed positivity
	property for large values of $n$, for instance up to 
	at least 10. In an upcoming paper, we deduce this fact from
	an explicit combinatorial formula for the Schur 
	expansion of $\chi_{\pi,l}$, which also impies the Loehr-Warrington
	formula \cite{loehr2008nested} and establishes
	the signed positivity of $\nabla^k s_{\lambda}$ in
	the monomial basis $m_\mu$.
	
	\section{Notations}
	For convenience of the reader, we summarize below some notation used throughout the paper.
	
	\begin{tabular}{cl}
		$G$ & $=GL_n$\\
		$T$ & $=(\C^*)^n\subset G$, the small torus\\
		$\wt T$ & $=T\times \C^* = (\C^*)^{n+1}$ the big torus\\
		$S$ & $=\C[x_1,\ldots,x_n]=H^*_T(\text{point})$\\
		$\tilde S$ & $=\C[x_1,\ldots,x_{n}, \epsilon]=H^*_{\tilde T}(\text{point})$ \\
		$F, \tilde F$ & the fraction field of $S$, $\tilde S$ respectively\\
		$S_n$ & the permutation group, i.e. the Weyl group of $G$\\
		$W$ & the group of affine permutations, i.e. the extended affine Weyl group\\
		$W_+$ & the monoid of positive affine permutations\\
		$Y, Y_k$ & the affine flag variety, the affine Springer fiber respectively\\
		$Z, Z_k$ & the positive part of $Y, Y_k$\\
		$H^*_T$ & the equivariant cohomology \\
		$H_*^T$ & the equivariant Borel-Moore homology.
	\end{tabular}

	\emph{Acknowledgments}. E. Carlsson was supported by
	NSF DMS-1802371 during part of this project. A. Mellit was supported by the projects Y963-N35 and
	P31705 of the Austrian Science Fund.
	
	\section{Combinatorial background and preliminaries}
	\label{sec:prelim}
	We recall some basic combinatorial background and notation,
	and also the main theorem of \cite{carlsson2020combinatorial},
	which is Theorem \ref{thm:oldthm} from the introduction.
	
	\subsection{Symmetric functions}
	
	\label{sec:sym}

	Let $\Sym$ denote the ring of symmetric
	functions, and write $\Lambda_R$ if we would like to
	specify a ground ring $R$, which by default is
	equal to $\mathbb{Q}$.
	We have the usual bases $s_\lambda,m_\lambda,h_{\lambda},
	e_{\lambda},p_{\lambda}$
	labeled by partitions $\lambda \in \Par(n)$,
	see \cite{macdonald1995symmetric}.
	For $f\in \Sym$,
	let $f[A]$ be the plethystic substitution
	homomorphism, which evaluates $f$
	at $p_k=A^{(k)}$, which is the result
	of substituting $v=v^k$ for every variable
	appearing in $A$, as in \cite{Haiman01vanishingtheorems}.

	If $V$ is a representation of $S_n$, we have
	its Frobenius character
	\[\frob V =
	\sum_{\lambda \in \Par(n)}\dim(V^{S_\lambda}) 
	m_\lambda \in \Sym.\]
	Here $S_{\lambda}=S_{\lambda_1}\times\cdots
	\times S_{\lambda_l}\subset S_n$ is the Young subgroup, and 
	$V^{S_{\lambda}}$ are the invariants.
	Note that both $S_{\alpha}$ and $m_\alpha$
	are also defined when $\alpha=(\alpha_1,...,\alpha_l)$ is
	only a strict composition of $n$, in other
	words ones of the $2^{n-1}$ finite sequences
	of positive integers summing to $n$, so that
	the sorted ones are just the partitions.

	If $V$ is a representation of 
	a product of $k$ factors of
	the symmetric group,
	we will denote the Frobenius character by 
	\[\frob_{X_1,...,X_k} V \in \C[x_{i,j}]^{S_n\times \cdots \times S_n},\]
	which is a function in $k$ sets of variables,
	$X_i=(x_{i,1},x_{i,2},...)$, individually symmetric in each one. 
	For modules with one or more gradings,
	which in this paper are bounded below in the negative degree,
	the Frobenius character will encode the
	degree in some generating variables
	$q_1,...,q_k$. For instance, for 
	bigraded modules which appear in this paper we
	will use the variables $q,t$, and so we will
	have
	\[\frob V=\sum_{i,j} q^it^j \frob V^{(i,j)}
	\in \Lambda_{\mathbb{Z}_{((q,t))}}\]
	where $V^{(i,j)}$ is
	the homogeneous component of the bigrading.
	For finitely generated modules, this can also be
	considered as an element of $\Lambda_{q,t}=\Lambda_{\mathbb{Q}(q,t)}$ by summing
	the generating function.
	
	Suppose $V=M$ is a graded module over a polynomial ring
	$\C[\xn]=\C[x_1,...,x_n]$, using bold letters for
	sets of variables, together with an action of $S_n$
	which intertwines the action on the variables 
	$\sigma(x_i)=x_{\sigma_i}$. In other words, $M$ is a 
	graded module over the smash product $\C[\xn]\ltimes S_n$.
	Then we have the Frobenius
	character $\frob M \in \Lambda_{q}$ regarding
	$M$ as a vector space. In this case, the Frobenius
	character of the Euler characteristic is determined
	by plethystic substitution. In other words,
	\begin{equation}
		\label{eq:plethysm}
		\sum_{i\geq 0} (-1)^i\frob \Tor_i^{\C[\xn]}(M,\C)=
		\frob M\left[X(1-q)\right]
	\end{equation}
	where $\C=\C[\xn]/(x_1,...,x_n)$ is the module on which
	$x_i$ acts by zero, the Frobenius character is taken with
	respect to the induced action of $S_n$ on the Tor groups,
	and for any symmetric function $F$ we have that
	$F[Z(1-q)]$ is the image of $F$ under the homomorphism
	defined on the power sum generators by
	\[p_i\mapsto p_i(z_1,z_2,...)(1-q^i),\quad Z=z_1+z_2+\cdots.\]
	Note that there is no relation between the symmetric function
	variables labeled by $z_i$ and the generators of the ring.
	We hope this does not create confusion when the symmetric function variables are given by the alphabet $X$.
	
	We also have the plethystic substitution 
	$F[-X]$ which is the result of substituting 
	$p_i(Z)\mapsto-p_i(Z)=(-1)^i \omega p_i(Z)$,
	where $\omega$ is the Weyl involution.
	
	\subsection{The nabla operator}
	
	\label{sec:nabla}
	
	Let $\Sym_{q,t}$ be the ring of symmetric
	over $\C(q,t)$. 
	Let $\Ht_\lambda=\Ht_\lambda(X;q,t)$ 
	denote the modified Macdonald polynomial
	\[\Ht_\lambda(X;q,t)=
	t^{n(\lambda)} J_{\lambda}[X/(1-t^{-1});q,t^{-1}],\]
	and let $\nabla$ be the operator of Bergeron-Garsia-Haiman-Tesler
	\cite{bergeron1999identities} defined by
	\begin{equation}
		\label{eq:nabladef}    
		\nabla \Ht_\lambda(X;q,t)=q^{n'(\lambda)} t^{n(\lambda)} \Ht_\lambda(X;q,t),
	\end{equation}
	where
	\begin{equation}
		\label{eq:macn}
		n(\lambda)=\sum_{i} (i-1)\lambda_i,\quad
		n'(\lambda)=n(\lambda')=\sum_{i} \binom{\lambda_i}{2}
	\end{equation}
	as defined in
	\cite{macdonald1995symmetric}. 
	Here $\lambda'$ is the transposed partition,
	and notice that $n'(\alpha)$ is defined for 
	compositions $\alpha$, not just partitions.
	When there is more than one symmetric
	function alphabet present, 
	we will suppose that $\nabla=\nabla_X$
	acts on the symmetric functions
	in the $X$ variables.
	
	Then the nabla positivity conjecture of
	Bergeron, Garsia, Haiman, Tesler
	is as follows:
	\begin{conjecture}[\cite{bergeron1999identities},
		Conjecture {\MakeUppercase {\romannumeral 1}}]
		\label{conj:nabpos}
		For any partitions
		$\lambda,\mu$ of norm $n$, and $k\geq 1$, we have that
		\begin{equation}
			\label{eq:nabpos}
			(-1)^{\iota(\lambda)}
			\left(\nabla^k s_{\lambda'},s_\mu\right)\in
			\mathbb{Z}_{\geq 0}[q,t],
		\end{equation}
		where the inner product is the Hall inner product,
		and
		\[\iota(\lambda)=
		\binom{l(\lambda)}{2}-
		\sum_{\lambda_i<i-1} (i-1-\lambda_i) = \sum_{i} \min(\lambda_i, i-1)
		\]
		is the number of cells of $\lambda$ below the main diagonal.
	\end{conjecture}

	\subsection{Standardization and super labels}
	
	Fix $n$ corresponding to degree,
	and define a label to be an
	$n$-tuple of positive integers
	$\an =(a_1,...,a_n)$ with $a_i\geq 1$. 
	We define a (strict) composition of $n$ by
	$\labtocmp(\an)=(\alpha_1,...,\alpha_{l})$,
	where $\alpha_i=\#\{j:a_j=c_i\}$,
	and $c_1<\cdots <c_l$ are the numbers that
	appear at least once in $\an$.
	For instance,
	$\labtocmp(2,2,4,1,4,2,1,5,4)=(2,3,3,1)$.
	We will set $X_{\an}=x_{a_1}\cdots x_{a_n}$.
	We define the $q$-automorphism factor
	\[\aut_q(\an)=\aut_q(\labtocmp(\an)),\quad \aut_q(\alpha)=\prod_i [\alpha_i]_q!\]
	
	Let $\mn\in \mathbb{Z}^n_{\geq 0}$,
	and write $|\mn|=m_1+\cdots+m_n$. We will be interested
	in $(d+1)$-tuples $(\mn,\an_1,...,\an_d)$ where
	each $\an_i$ is a label. We define
	the class $[\mn,\an_1,...,\an_d]$ to be the orbit
	of the diagonal action of $S_n$ by simultaneously 
	permuting the entries. We will identify these orbits
	with their minimal representatives, for which $\mn$
	is in decreasing order, and each $\an_i$
	is in increasing order when there is a tie, lexicographically. Such tuples are called \emph{sorted}.
	We similarly have $\alpha=\alpha(\mn,\an_1,...,\an_d)$
	defined so that $S_\alpha$ is the stabilizer subgroup
	for the diagonal action of $S_n$, and similarly for
	$\aut_q(\mn,\an_1,...,\an_d)$.
	For instance, the sorted representative of
	\[[\mn,\an,\bn]=[(2,1,2,2,0,2,0),(2,2,2,2,3,1,3),(1,3,1,1,1,2,1)]\]
	is $((2,2,2,2,1,0,0),(1,2,2,2,2,3,3),(2,1,1,1,3,1,1))$.
	Then we have 
	$\alpha(\mn,\an,\bn)=(1,3,1,2)$, and 
	\[\aut_q(\mn,\an,\bn)=[1]_q![3]_q![1]_q![2]_q!=
	(1+q)^2(1+q+q^2).\]
	
	\label{sec:super}
	We briefly recall some facts
	from \cite{haglund2005combinatoriala}, Section 4.
	First, we have the super alphabet
	\[\mathcal{A}=\Z_+\cup \Z_-=
	\left\{1,\bar{1},2,\bar{2},...\right\}.\]
	The integers $i$ are called positive, while
	the overlined numbers $\bar{i}$ are called negative.
	The set is totally ordered by
	$\{1<\bar{1}<2<\bar{2}<\cdots\}$,
	which is referred to there as $<_1$.
	Let the negative of a letter be the result of
	reversing the sign, and let absolute value 
	remove any sign leaving
	only the positive part. In other words,
	$-c=\overline{c}$, $-\overline{c}=c$, and
	$|c|=|\overline{c}|=c$
	for $c \in \mathbb{Z}_{\geq 1}$.
	We set $X_{\overline{a}}=X_{a}$.
	
	We will make use of the following definition
	\begin{defn}
		\label{def:standardization}
		The standardization $\sigma=\shuff(\an)$
		of a label $\an\in \mathcal{A}^n$
		is the unique permutation $\sigma$ such that
		$\an_{\sigma^{-1}}$ is weakly increasing, and
		the restriction of $\sigma$ to $\an^{-1}(\{x\})$
		is increasing if $x$ is positive, decreasing
		if $x$ is negative. Here we are viewing
		$\an$ as a function $\{1,...,n\}\rightarrow \mathcal{A}$.
	\end{defn}
	For instance,
	$\shuff((2,\bar{1},1,4,2,\bar{1},\bar{1}))=
	(5,4,1,7,6,3,2)$.
	
	\subsection{Previous results}
	
	\label{sec:prev}
	
	We recall the results from \cite{carlsson2020combinatorial},
	which we refer to for more details.
	
	We begin by recalling the $\dinv$ statistic.
	\begin{defn}
		\label{def:dinv}
		Let $\mn\in \mathbb{Z}_{\geq 0}^n$,
		let $\an,\bn$ be labels, let $k\geq 1$,
		and suppose that $(\mn,\an)$ are sorted. We define
		\begin{subequations}
			\label{eq:dinvdef}
			\begin{equation}
				\label{eq:dinvdefa}
				\dinv_k(\mn,\an,\bn)=\sum_{i<j} \dinv^{i,j}_k(\mn,\an,\bn)
			\end{equation}
			where
			\begin{equation}
				\label{eq:dinvdefb}    
				\dinv_k^{i,j}(\mn,\an,\bn)=\max\left(
				m_j-m_i-1+k+\epsilon(a_i,a_j)+\epsilon(b_i,b_j),0\right),
			\end{equation}
		\end{subequations}
		and $\epsilon(a_1,a_2)$ is one if $a_1>a_2$, zero otherwise.
	\end{defn}
	For just one label $\an$, we define $\dinv_k(\mn,\an)$ as the result of removing
	$\epsilon(b_i,b_j)$ from \eqref{eq:dinvdefb}, 
	which has the same effect as setting $\bn=(1^n)$.
	
	Recall that an $n\times n$ Dyck path $\pi$ is a path
	in the $n\times n$ grid starting at $(0,0)$
	and ending at $(n,n)$,
	traveling only North and East, and never
	crossing below the main diagonal \cite{haglund2008catalan}.
	It will often be denoted with 1's signifying North
	steps and 0's for East steps, so that the path
	in Figure \ref{fig:dyckpathalpha} would be
	given by $\pi=111001101000$. 
	A Dyck path is determined uniquely by the set
	\[\uio(\pi)=\left\{
	(i,j) : \text{$1\leq i<j\leq n$
		is between the path and the diagonal}
	\right\}.\]
	We will define a partial Dyck path \cite{carlsson2018proof}
	as a pair $(\pi,l)$ where $l$ is an most the number of trailing
	East steps in $\pi$. We will write them by simply leaving off
	that many zeros from the end of $\pi$, in other words
	write $1110011010$ instead of $(\pi,2)$
	for the above Dyck path $\pi$.
	
	It is not hard to show that the following definition
	defines a Dyck path:
	\begin{defn}
		\label{defn:dinv1}
		Fix $k\geq 0$, 
		suppose $(\mn,\an)$ is sorted, and let $i<j$.
		We will say that $i$ attacks $j$ if
		\[m_j-m_i-1+k+\epsilon(a_i,a_j)\geq 0.\]
		Let $\dyckpath=\dyckpath_k(\mn,\an)$ denote the Dyck path such that the elements of 
		$\uio(\dyckpath)$, 
		are the pairs $i<j$ for which 
		$i$ attacks $j$.
	\end{defn}
	We now have that
	\begin{equation}
		\label{eq:dinvdp}
		\dinv_k(\mn, \an, \bn) = 
		\dinv_k(\mn,\an)+\inv_{\dyckpath_k(\mn,\an)}(\bn)
	\end{equation}
	where
	\begin{equation}
		\label{eq:invpi}
		\inv_\dyckpath(\bn)  =
		\#\{(i,j)\in \uio(\dyckpath) : b_i>b_j\}.
	\end{equation}
	We can also determine the composition
	$\alpha(\mn,\an)$ from $\pi_0(\mn,\an)$.
	Let $\alpha=\pitoalpha(\pi)$ be defined by
	drawing lines through all edges, and recording
	the points where they contact the diagonal.
	In other words, $\alpha$ is the composition whose
	endpoints are the result of intersecting the vertical
	and horizontal lines with the diagonal.
	Then it is can be checked that $\alpha(\pi_0(\mn,\an))=\alpha(\mn,\an)$. For arbitrary $k\geq 0$, the composition $\alpha(\mn,\an)$ is a refinement of $\alpha(\pi_k(\mn,\an))$.

	\begin{figure}
		\begin{centering}
			\begin{tikzpicture}
				\draw[help lines] (0,0) grid (6,6);
				\draw[dashed,color=gray] (0,0)--(6,6);
				\draw[-,thick,color=gray] (3,5)--(5,5);
				\draw[-,thick,color=gray] (3,5)--(3,3);
				\draw[-,thick,color=gray] (2,3)--(3,3);
				\draw[-,thick,color=gray] (2,3)--(2,2);
				\draw[-,very thick] (0,0)--(0,1);
				\draw[-,very thick] (0,1)--(0,2);
				\draw[-,very thick] (0,2)--(0,3);
				\draw[-,very thick] (0,3)--(1,3);
				\draw[-,very thick] (1,3)--(2,3);
				\draw[-,very thick] (2,3)--(2,4);
				\draw[-,very thick] (2,4)--(2,5);
				\draw[-,very thick] (2,5)--(3,5);
				\draw[-,very thick] (3,5)--(3,6);
				\draw[-,very thick] (3,6)--(4,6);
				\draw[-,very thick] (4,6)--(5,6);
				\draw[-,very thick] (5,6)--(6,6);
			\end{tikzpicture}
			\caption{The Dyck path $\pi=111001101000$. The boxes
				between the path and the diagonal are
				$\uio(\pi)=\{(1,2),(1,3),(2,3),(3,4),(3,5),(4,5),(4,6),(5,6)\}$, and we have $\pitoalpha(\pi)=(2,1,2,1)$.}
			\label{fig:dyckpathalpha}
		\end{centering}
	\end{figure}
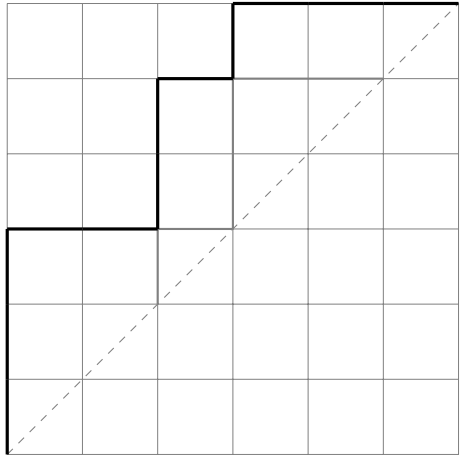

	We have an LLT polynomial 
	\begin{equation}
		\label{eq:csfa}
		\csfa_{\dyckpath}[Y;q]=
		\sum_{\bn} q^{\inv_{\dyckpath}(\bn)}Y_{\bn}.
	\end{equation}
	As we explain below, this polynomial appears as the
	Frobenius character of the equivariant cohomology
	of the regular semisimple Hessenberg variety associated
	to $\pi$, and is the plethystically transformed
	chromatic symmetric function of Stanley, which was proved in \cite{brosnan2018hessenberg,shareshian2012chromatic}.
	It is also the LLT polynomial that appears in \cite{carlsson2018proof}.

	The main theorem of \cite{carlsson2020combinatorial}
	states 
	\begin{thm}
		\label{thm:oldthm1}
		For any $k\geq 1$, we have
		\[\nabla_X^k e_n\left[\frac{XY}{(1-q)(1-t)}\right] 
		= \sum_{[\mn, \an, \bn]}  \frac{t^{|\mn|} 
			q^{\dinv_k(\mn, \an,\bn)}}{(1-q)^n \aut_q(\mn,\an,\bn)}
		X_{\an} Y_{\bn}\]
		\begin{equation}
			\label{eq:oldnabla}
			=  \sum_{[\mn,\an]} \frac{t^{|\mn|}q^{\dinv_k(\mn,\an)}}{(1-q)^n
				\aut_q(\mn,\an)}X_{\an} \csfa_{\pi_k(\mn,\an)}[Y;q].
		\end{equation}
	\end{thm}
	
	We extend Definition \ref{def:dinv} to super alphabets by setting
	\begin{equation}\label{eq:oddepsdef}\epsilon(b_1,b_2)=\begin{cases} 1 & 
			\mbox{$b_1 >_1 b_2$ or $b_1=b_2$ is negative, or} \\
			0 & \mbox{otherwise}.
		\end{cases}
	\end{equation}
	We similarly replace
	\[\inv_\pi(\bn)=\sum_{(i,j)\in \uio(\pi)} \eps(b_i,b_j).\]
	Notice that $\pi_k(\mn,\an)$ makes sense for super
	alphabets as well. Then
	\begin{cor}
		\label{cor:supthm}
		Theorem \ref{thm:oldthm1} holds for super alphabets, substituting
		\eqref{eq:oddepsdef} into $\dinv_k$ for negative labels.
	\end{cor}
	
	\begin{proof}
		The result holds for super labels in $\bn$
		using equation \eqref{eq:oldnabla}, and the desired 
		statement about $\csfa_\pi[Y;q]$. We then apply the
		symmetry between the $X$ and $Y$ variables.
		
	\end{proof}
	
	\begin{example}
		Let $\pi=110100$. Then we have
		\[\csfa_{\pi}=
		m_{{3}}+ \left( 2\,q+1 \right) m_{{2,1}}+ \left( {q}^{2}+4\,q+1
		\right) m_{{1,1,1}}.\]
		We can see that the coefficient of $m_{2,1}$ in $\omega \csfa_{110100}$ is
		\[q^{\dinv_\pi((\bar{1},\bar{1},\bar{2}))}+q^{\dinv_\pi((\bar{1},\bar{2},\bar{1}))}+
		q^{\dinv_\pi((\bar{2},\bar{1},\bar{1}))}
		=q^2+2q.\]
	\end{example}
	
	\subsection{Affine permutations}
	\label{sec:affperm}
	
	We have the group of (extended) affine permutations
	\begin{equation}
		\label{eq:affperm} 
		W=\left\{w: \mathbb{Z}\rightarrow \mathbb{Z}:
		w_{i+n}=w_i+n\right\}.
	\end{equation}
	Each such permutation is uniquely determined by its
	values on the elements $\{1,...,n\}$, and we will
	often describe it elements in window notation
	$w=(w_1,...,w_n)$.
	We have a homomorphism $W\rightarrow \Z$
	given by $w\mapsto d$, where $d$ is the unique number satisfying $w_1+\cdots+w_n=dn(n+1)/2$,
	whose kernel $W_0$ is the usual group of affine permutations, which is the affine Weyl group 
	in the $SL_n$ case. 
	We then have the decomposition into $W_0$ orbits
	\[W=\bigsqcup_{d\in \Z} W_d,\quad W_d=W_0 \rho^d=\rho^dW_0,\]
	where $\rho\in W$ is the rotation element
	$\rho_i=i+1$.
	
	The Bruhat order on $W$ is defined as the usual order
	on $W_0$ which is generated by reflections, and so that
	elements in different $W_d$ components are incomparable,
	and $v\rho^d\leq w\rho^d$ if and only if $v\leq w$.
	We have the usual inversion statistic
	\[\inv(w)=\sum_{1\leq i<j\leq n} \inv^{i,j}(w),\quad
	\inv^{i,j}=\max\left(
	\lceil (w_i-w_j)/n \rceil,
	\lceil(w_j-w_i)/n-1\rceil\right),\]
	which is independent of $\rho$,
	$\inv(w\rho^d)=\inv(\rho^dw)=\inv(w)$.

	If $J\subset \{0,...,n-1\}$ is nonempty,
	we have the finite subgroups $W_J\subset W$ generated by $s_i$ for $i\notin J$.
	The subsets $J\subset \{1,...,n-1\}$
	are in bijection with
	compositions $\alpha=(\alpha_1,...,\alpha_l)$ of $n$
	in such a way that $j \in J$ if $j,j+1$ are in different blocks of $\alpha$,
	and in this case $W_J=S_\alpha$ is just the Young
	subgroup of $S_n\subset W$. The other $W_J$
	are conjugates of the Young subgroups
	by the rotation element $\rho$.
	We have a Bruhat order on arbitrary double
	cosets $W_{J}\backslash W/W_{J'}$, which is
	the same as the order induced by the map
	of posets $W_J\backslash W/W_{J'}\subset W$
	which selects the unique minimal representative
	$w_-$ of each coset, or alternatively the unique maximal one $w_+$, which will be denoted
	$w_{\pm}\in W_{J}\backslash W/W_{J'}$.

	We now explain a correspondence between affine permutations and the combinatorial statistics of the last section. Let $W_+$ denote the set 
	of \emph{positive} affine permutations
	\begin{equation}
		\label{eq:wbij}
		W_+=\left\{w\in W:
		w(\mathbb{Z}_{\geq 1})\subset \mathbb{Z}_{\geq 1}\right\},
	\end{equation}
	in other words those affine permutations whose window
	notation contains only positive integers.
	In \cite{carlsson2020combinatorial}, 
	we considered a bijection 
	\[\paff:\{[\mn,\an,\bn]: X_\an Y_\bn=X^\alpha Y^\beta \}\longleftrightarrow
	S_\beta \backslash W_+ / S_\alpha.\]
	for each pair of compositions $\alpha,\beta$.
	It is defined by $\paff(\mn,\an,\bn)=S_\beta \paff_0(\mn,\an,\bn) S_\alpha$, where
	\begin{equation}
		\label{eq:paff}
		\paff_0(\mn,\an,\bn)=
		\shuff_{>}(\revl(\bn)) \trans{\mn}
		\shuff_{<}(\an)^{-1}
	\end{equation}
	is a maximal representative of the double coset.
	Here $\trans{\mn}=(n+m_1 n,...,1+m_nn)$ is the maximal representative of its coset in $S_n\backslash W_+ /S_n$,
	and $\revl(\bn)$ is the result of writing
	$\bn$ in the reverse order. 
	If no $\bn$ label is specified, we will define
	$\paff(\mn,\an)=\paff(\mn,\an,(1^n))\in 
	S_n\backslash W_+/S_\alpha$.
	
	We can now translate some of the constructions from
	the last section in terms of the corresponding
	double coset under $\paff$. 
	First, we can recover the $\dinv$ statistic
	by $\dinv_k(\mn,\an,\bn)=\dinv_k(w_+)$ where
	$w_+=\paff_0(\mn,\an,\bn)$, and 
	\begin{equation}
		\label{eq:dinvw}
		\dinv_k(w) =
		\sum_{1\leq i<j\leq n}
		\max\left(k-\inv^{i,j}(w^{-1}),0\right)
		= k \binom{n}2 - \sum_{1\leq i<j\leq n}
		\inv_k^{i,j}(w^{-1}),
	\end{equation}
	where $\inv_k^{i,j}(w^{-1})=\min(k,\inv^{i,j}(w^{-1}))$,
	which also holds if there is no third label, 
	$\bn=(1^n)$.
	We also have a rule
	\begin{equation}
		\dinv_k(\mn,-\an,\bn)=\dinv_k(w)
	\end{equation}
	where $w$ is the minimal element in the
	maximal right coset
	$w_+S_\alpha \subset \paff(\mn,\an,\bn)$,
	for $w_+=\paff_0(\mn,\an,\bn)$.
	The composition 
	$\alpha'=\alpha(\mn,\an)$ 
	is the unique one with the property that
	$S_{\alpha'}$ 
	is the stabilizer of $S_n$
	acting on $w_+S_\alpha$.
	Finally, the Dyck path 
	$\pi=\pi_k(\mn,-\an)$
	is the one determined by
	\begin{equation}
		\uio(\pi)=\left\{(i,j):
		\inv^{i,j}(w_-^{-1})< k\right\}
		\label{eq:wtodp}
	\end{equation}
	where $i<j$ are in different blocks of $\alpha'$,
	where $w_-$ is the minimal element of 
	$\paff(\mn,\an)$ in $S_n \backslash W_+/S_\alpha$.
	It will be denoted $\pi_k(S_n w S_\alpha)$.
	
	\section{GKM spaces}
	
	We review some necessary background on 
	equivariant Borel-Moore homology and
	GKM spaces. We begin wih the general setup,
	and then present the relevant examples.
	
	\subsection{General setup}
	\label{sec:gkm}
	Let $X$ be a (possibly singular)
	complex projective variety variety together
	with an action of an algebraic torus $T=(\C^*)^d$.
	Suppose that $X$ satisfies the Goresky-Kottwitz-Macpherson
	(GKM) conditions, namely that the fixed point set
	$X^{T}$ is finite, there are finitely
	many one-dimensional orbits, and $X$
	is equivariantly formal.
	The theorem of 
	Goresky-Kottwitz-Macpherson 
	\cite{goresky1997koszul} 
	says that the equivariant 
	cohomology $H^*_T(X)$, which in this
	paper has complex coefficients, 
	injects into the cohomology
	of the fixed point set
	$H^*_T(X)\hookrightarrow H^*_T(X^T)$, and that
	the image can be described by algebraic conditions
	determined by the weighted moment graph.

	Let $H_*^T(X)$ be the equivariant Borel-Moore
	homology, as defined in \cite{brion1998equivariant,edidin96equivariant}.
	Then as in these references, we have that
	$H_*^T(X)$ is free over 
	$H_T^*(pt)=\C[t_1,...,t_d]$ which will
	be denoted by $S$,
	and the localization map
	$i_*:H_*^T(X^T)\rightarrow H_*^T(X)$ becomes
	an isomorphism after inverting finitely many
	characters of $T$. The homology may
	then be identified as an $S$-submodule of the localization
	\[H_*^T(X^{T})\otimes_{S}
	F= F\cdot X^T=
	\bigoplus_{p \in X^T} F\cdot p\]
	via $i_*^{-1}$, where $F=\C(t_1,...,t_d)$
	is the field of fractions of $S$.
	Note that the homological grading is
	in the negative direction, 
	so that multiplication by $S$ action raises degree.
	Suppose that $X$ is paved by 
	$T$-invariant affines, meaning that we have
	$X=\bigsqcup_i U_i$, where $U_i \cong \mathbb{A}^{d_i}$
	for some $d_i$, and
	\begin{equation}
		\label{eq:indvariety}
		X_i=\overline{U}_i\subset
		\bigcup_{j\leq i} U_j
	\end{equation}
	for some partial order on the indices. 
	Then by \cite{graham2001positivity}
	Proposition 2.1, the fundamental classes 
	$\hovar_i=[X_i]_T$ freely generate $H_*^T(X)$
	as a module over $S=H^T_*(pt)$.
	By the above assumptions about $X$
	and the same proposition, there are dual
	generators $\covar_i \in H_T^*(X)$
	satisfying $\covar_i(\hovar_j)=\delta_{i,j}$.

	Let $G=(V,E,\wf)$ be the moment graph
	of $X$, where $V=X^T$ is the fixed point
	set, $E$ is the set of one-dimensional
	orbits, and $\wf:E\rightarrow S$ 
	is the weight function of the action.
	In general, the edge set can directed in
	more than one way by
	choosing a sufficiently
	generic one-dimensional
	subtorus $T'\subset T$,
	and deciding that the edge points
	towards the South pole of the corresponding
	one-dimensional orbit, where the South
	pole is determined as the attracting
	point with respect to $T'$, see
	\cite{guillemin2006description,
		tymoczko2005introduction,
		tymoczko2008permutation,
		guillemin2003existence}. This determines a
	potentially different partial order on $V$.
	For instance, while the Bruhat order is the standard
	one to use for flag varieties, one may also consider the
	conjugates of this order by permutations, which will
	be an important observation in Section \ref{sec:thmb}.

	The homological and cohomological generators
	then satisfy
	\begin{equation}
		\hovar_p=\sum_{q\leq p} c_{p,q} q,\quad
		\covar_p=\sum_{q\geq p} d_{p,q} q
	\end{equation}
	as elements of $F\cdot V$. Moreover,
	the leading terms are given by
	\[c_{p,p}=d_{p,p}^{-1},
	\quad d_{p,p}=\lt(p)=
	\prod_{e\in \outedges(p)} 
	\wf(e),\]
	where $\outedges(p)=\{(p,q):q\in \outneighbors(p)\}$ 
	is the set
	of outgoing edges, and $\outneighbors(p)$ is the
	set of endpoints. The leading term
	follows from combining the GKM relations
	with the triangularity, and in some cases 
	the other coefficients can 
	be deduced using the type of argument 
	used by Knutson and Tao in \cite{knutson2001puzzles}. 
	The classes $\covar_p$ are often referred to as canonical classes, and if some conditions known as
	the Palais-Smale conditions on the GKM
	data are satisfied, they are unique \cite{guillemin2003existence}.
	One condition that forces the uniqueness is
	if the degree strictly increases along increasing
	chains in the poset determined by $(V,E)$,
	which is true in the case of the affine
	flag variety, but not the regular semisimple Hessenberg
	variety or the unramified affine Springer fiber $Y_k$ defined below.
	
	\subsection{The affine flag variety}
	\label{sec:affgkm}
	
	We recall the construction of the affine flag variety,
	referring to \cite{kumar2002kac,lam2010affine,lam2014schur} for more details.
	
	Let $\mathcal{K}=\C((z))$ and $\mathcal{O}=\C[[z]]$,
	and consider the groups $G(\mathcal{O})\subset G(\mathcal{K})$
	for $G=GL_n$.
	We have the affine Weyl group,
	which is the set of affine permutations from \eqref{eq:affperm}.
	It acts on $\mathcal{K}^n$ by $w\cdot e_j=e_{w_j}$
	for $j\in \Z$,
	where $e_{i+dn}=z^{-d} e_i\in \mathcal{K}^n$ for 
	$i\in \{1,...,n\}$, identifying
	it as a subgroup $W\subset G(\mathcal{K})$.
	If $\alpha$ is a composition, let 
	$P_\alpha\subset GL_n$
	denote the corresponding parabolic subgroup,
	so that $P_{(1,...,1)}=B$, the upper triangular
	matrices.
	If $J\subset \{0,...,n-1\}$ is nonempty
	as in Section \ref{sec:affperm}, we
	have the parahoric subgroup
	$I_J\subset G(\mathcal{K})$ 
	containing the group $W_J$, so that
	$I_{\{0,...,n-1\}}$ is the usual Iwahori
	subgroup, and $I_{\{0\}}=G(\mathcal{O})$.
	If $J\subset \{1,...,n-1\}$ corresponds
	to the composition $\alpha$, then
	$I_J$ is the inverse image of $P_\alpha$ 
	under the map
	$\pi:G(\mathcal{O})\rightarrow G$ which 
	evaluates at $z=0$. The others may be
	obtained by conjugating these subgroups by $\rho$.

	Let $Y^J=G(\mathcal{K})/I_J$ denote the affine flag variety.
	If no parabolic is specified,
	let $Y=Y^{\{0,...,n-1\}}$ denote the one corresponding
	to the standard Iwahori subgroup, i.e. the full flag variety,
	and denote $X=Y^{\{0\}}$ for the affine Grassmannian.
	Then parametrically we have that
	$X=\left\{V\subset \mathcal{K}^n\right\}$
	where $V$ is an $\mathcal{O}$-submodule, and there
	exists $N>0$ so that
	$z^{-N} \mathcal{O} \subset V \subset z^N\mathcal{O}$.
	The full flag variety may be described by
	\[Y=\left\{V_0\subset \cdots \subset V_n: V_0=zV_n\right\},\]
	where each $V_i\in X$, and $\dim(V_{i+1}/V_i)=1$.
	We may also consider a flag as periodic and infinite
	by setting $V_{i}=zV_{i+n}$, and so refer to $V_i$
	for any $i\in \Z$.
	The image of a Weyl group element $w\in W$ 
	determines a flag by
	$V_i= \langle e_{w_j}: j\leq i\rangle$
	where $e_i$ was defined above.
	Each $Y^J$ decomposes as a disjoint union
	\[Y^J=\bigsqcup_{d\in \Z} Y_{SL_n}^{\rho^d(J)},\]
	where $Y^J_{SL_n}$ is the partial affine flag variety for $SL_n$.
	Each $\rho^d(J)$ is obtained by applying $\rho$ to the
	elements of $J$ modulo $n$, and the Springer action is
	conjugated by $\rho$ in the compatible way.

	We now describe the equivariant homology and
	cohomology of $Y^{J}$.
	There is an action on $Y^J$ by the extended torus
	$\tilde{T}\cong (\mathbb{C}^*)^{n}$, which is the small
	torus $T\subset GL_n$ together with loop rotation.
	The ground ring, which is the equivariant cohomology of
	a point with respect to $T,\tilde{T}$ are given by 
	$S=\C[x_1,...,x_n]$ and 
	$\tilde{S}=\C[x_1,...,x_n,\eps]$ respectively,
	where $x_1,...,x_{n}$ are the weights of $T$, 
	and $\eps$ is the parameter corresponding
	to loop rotation. We have the root elements
	\[\alpha_{i}=\alpha_{i,i+1},\quad 
	\alpha_{i,j}=\lambda_i-\lambda_j,\quad \lambda_i=x_{\bar{i}}+\frac{i-\bar{i}}{n}\eps\]
	where $\bar{i}$ is the unique element of $\{1,...,n\}$
	which is congruent to $i$ modulo $n$. 
	Then $\awgsl$ acts by automorphisms of $\tilde{S}$
	in such a way that $w(\lambda_i)=\lambda_{w_i}$.
	
	The fixed points of both $T,\tilde{T}$ are given by
	$V=W/W_J$ via the inclusion described above,
	and the corresponding point will be labeled $p_w$.
	The edges in the GKM graph are determined by
	\begin{equation}
		\label{eq:outneighborsflag}
		\outneighbors(w)=
		\left\{v \bruleq w : w=t_{i,j}v\right\}.
	\end{equation} 
	for some affine transposition $t_{i,j}$,
	which is an element of $\awgsl$
	that switches two elements $i,j\in \mathbb{Z}$
	which are not congruent modulo $n$,
	and is necessarily in $W_0$.
	We have the corresponding definition 
	$v,w$ are replaced
	by cosets $vW_J,wW_J$, and in fact is also the same
	as replacing the cosets by minimal representatives
	in the Bruhat order.
	The weight function is determined by
	\begin{equation}
		\label{eq:affwf}
		\wf(t_{i,j}w,w)=\alpha_{i,j},
	\end{equation}
	whose degree is the affine inversion number $\inv(w)$ defined in Section \ref{sec:affperm}.

	The cells in the affine paving 
	\eqref{eq:indvariety} are
	the Schubert cells $\Omega^J_w=I w I_J$.
	The Schubert varieties are the closures
	\[V_w=\overline{\Omega}_w=\bigsqcup_{v\leq w} \Omega_v.\]
	The homology is defined by
	treating $Y^J$ as an ind-variety using
	the Bruhat filtration, see \cite{kumar2002kac}.
	If $P\subset W/W_J$ is a finite lower set, meaning a 
	finite subset for which 
	\[v\leq w,\ w\in P \Rightarrow v\in P,\]
	then the subspace
	\[V^J_P=\bigcup_{u W_J \in P} V^J_{u W_J}=
	\bigsqcup_{uW_J\in P} \Omega^J_{u W_J}\] 
	is a closed subvariety of $Y^J$ which is
	GKM with respect to the action of $\tilde{T}$.
	We then have the direct limit
	\begin{equation}
		\label{eq:affhomologylim}    
		H_*^{\tilde{T}}(Y^J)=
		\lim_{\rightarrow} H_*^{\tilde{T}}
		\left(V^J_P\right)
	\end{equation}
	over all lower sets $P\subset W/W_J$.
	
	We have commuting left and right actions called \emph{dot} and
	\emph{star} of $W$ on $H_*^{\tilde{T}}(Y)$
	induced by the action by multiplication on $F\cdot \awgsl$
	via $i_*^{-1}$, in which the dot action
	acts on the ground field $\tilde{F}$ by the automorphisms
	described above. 
	Moreover, we can recover the homology of both the $GL_n$
	version $Y$ and the parabolics $Y^J$ from the 
	$SL_n$ case $Y_{SL_n}$, via the formula
	\begin{equation}
		\label{eq:affhat}
		H_*^{\tilde{T}}
		(Y^{J})=
		\bigoplus_d \rho^d H_*^{\tilde{T}}(Y_{SL_n}^{J}),\quad
		H_*^{\tilde{T}}(\affsl^J)\cong
		H_*^{\tilde{T}}(\affsl)^{W_J},
	\end{equation}
	the second isomorphism being one of the 
	fundamental properties of the Springer action.
	Here $M^{W_J}$ are the invariants with respect 
	to $W_J\subset W_0$ 
	acting on the right by the star 
	action twisted by the star action,
	which is defined by identifying
	$W_J=\rho^k S_{\alpha} \rho^{-k}$ for some $k$.
	The inverse of the corresponding generators 
	$A_{w S_\alpha}^J$ are the result of symmetrizing
	$A_w$ by $W_J$, when $w$ is minimal in the 
	coset $W/W_J$.

	Though we will not need
	this definition, the cohomology is defined as
	the subset of the inverse limit
	\begin{equation}
		\label{eq:indhomology}
		H^*_{\tilde{T}}(X)\subset \lim_{\leftarrow} H^*_{\tilde{T}}(X_w)
	\end{equation}
	as graded $\tilde{S}$-modules \cite{graham2001positivity}.
	Equivalently, one may take cohomology
	to be all those elements $\covar$ 
	in the (non graded)
	inverse limit such that $\covar(A_w)=0$
	for all but finitely many $w$.
	For $\affsl$, this agrees with 
	the first definition, 
	and is more useful when the cohomology
	is defined algebraically as the dual
	module to the nil Hecke algebra
	\cite{lam2014schur}, defined
	in Section \ref{sec:nilhecke} below.
	In either case, there are canonical elements
	$\covar_v\in H_T^*(Y)$ which restrict
	to $\covar_v$ on each $H_T^*(V_w)$, 
	and which are the Kostant-Kumar basis \cite{kostant1986nil}.

	\subsection{The nil Hecke ring}
	
	\label{sec:nilhecke}
	We now recall the nil Hecke algebra $\aaf$ of Kostant and Kumar, which is an algebraic presentation of
	$H_*^{\tilde{T}}(Y)$, for which we refer to  
	\cite{kostant1986nil,lam2014schur}.
	Let $\tilde{F}=\C(\xn,\epsilon)$ be the fraction field
	field of $\tilde{S}$. Then we have a 
	left and a right action of $W$ on the free 
	$\tilde{F}$-vector space $\tilde{F}\cdot W$, where the
	left action intertwines the scalars by
	\[w f(\lambda_1,...,\lambda_n)= f(\lambda_{w_1},...,\lambda_{w_n})w,\]
	and the weights $\lambda_i$ are the ones
	defined above.

	Now consider the action of the operators
	\[A_i=\frac{1-s_i}{\alpha_i} \in \End\left(\tilde{F}\cdot W\right),\]
	for $i\in \{0,...,n-1\}$,
	where $s_i\in \awgsl$ 
	is the simple transposition.
	They define the action of Kostant and Kumar's
	nil Hecke ring $\mathbb{A}_{af}$ \cite{kostant1986nil}
	in type A, which 
	is the $\C$-algebra generated by 
	$(A_0,...,A_{n-1})$ and $\tilde{S}$,
	subject to the following relations:
	\begin{align}
		\label{eq:aafrel}
		A_i \lambda = (s_i\cdot \lambda)A_i+
		(\alpha_i^{\vee},\lambda)\cdot 1,& \nonumber \\
		A_iA_i=0,& \nonumber \\
		A_i A_j A_i\cdots=A_j A_i A_j \cdots , &\mbox{ if }
		s_i s_j s_i\cdots=s_j s_i s_j\cdots.
	\end{align}
	Here $\alpha_i^{\vee}$ is the dual root.
	For more details, see 
	\cite{lam2014schur} Chapter 3, Section 6.
	
	The following summarizes the connection between
	$\aaf$ and the homology of the affine flag variety
	in type A:
	\begin{prop}
		\label{prop:aaf}
		If $w=s_{i_1}\cdots s_{i_l}\rho^d\in W$ is a reduced word decomposition, then we have that
		$\varphi(A_{i_1}\cdots A_{i_l}\rho^d)=A_w$ 
		under the isomorphism
		$\varphi:\tilde{F} \cdot W \rightarrow \tilde{F}\cdot V$ which sends $w$ to the fixed point $p_w$.
		In particular, we have an isomorphism
		\[H_*^{\tilde{T}}(Y)=\bigoplus_d \rho^d \aaf\]
		of $\tilde{S}$
		modules which intertwines the left and right
		action of $W$.
	\end{prop}
	The result also holds with $T$ in place of $\tilde{T}$,
	though in this case the flag variety is not GKM.
	Using \eqref{eq:affhat}, we may recover the homologies
	of $H_*^{\tilde{T}}(Y^J)$.
	
	\begin{example} 
		A typical basis element 
		\[A_{(1,6,2)}= A_{2}A_{1}A_{0} \rho^1\in \aaf\rho^1\subset H^{\tilde{T}}_*(Y)\] 
		is given by
		\begin{align}
			&\nonumber
			{\frac {p_{(3,2,4)}}{ \left(\eps+x_{{1}}-x_2 \right)  \left( x_{{1}}
					-x_{{3}} \right)  \left( x_{{2}}-x_{{3}} \right) }}+{\frac {p_{(2,4,3)}
				}{ \left(\epsilon+x_{{1}} -x_3\right)  \left( x_{{1}}-x_{{2}} \right) 
					\left( x_{{2}}-x_{{3}} \right) }}-\\
			\nonumber & {\frac {p_{(2,3,4)}}
				{\left(\epsilon+x_{{1}} -x_3\right)  \left( x_{{1}}-x_{{2}} \right)  \left( x_{{2}}-x_{
						{3}} \right) }}+{\frac {p_{(1,6,2)}}{ \left(\epsilon+x_{{3}}-x_2
					\right)  \left( x_{{1}}-x_{{3}} \right)  \left( x_{{2}}-x_{{3}}
					\right) }}-\\
			\nonumber & {\frac {p_{(1,5,3)}}{ \left(\epsilon+x_{{2}} -x_3\right)
					\left( x_{{1}}-x_{{2}} \right)  \left( x_{{2}}-x_{{3}} \right) }}+
			{
				\frac {p_{(1,3,5)}}{ \left(\eps+x_{{2}} -x_3\right)  
					\left( x_{{1}}-
					x_{{2}} \right)  \left( x_{{2}}-x_{{3}} \right) }}-\\
			\label{eq:ak162} & {\frac {p_{(1,2,6)}
				}{ \left(\eps+x_{{3}}-x_2 \right)  \left( x_{{1}}-x_{{3}} \right) 
					\left( x_{{2}}-x_{{3}} \right) }}-{\frac {p_{(3,4,2)}}{ \left( -x_{{2}
					}+\eps+x_{{1}}\right)  \left( x_{{1}}-x_{{3}} \right)  \left( x_{{2}}-x_{
						{3}} \right) }}.\nonumber
		\end{align}
		Notice that $(1,6,2)$ is minimal in 
		$W/W_{\{2\}}$, so that $A_{(1,6,2)}$
		maps to a generator of 
		$H_*^{\tilde{T}}(Y^{\{2\}})=
		H_*^{\tilde{T}}(Y^{\{0\}})\rho$.
		On the other hand, it maps to zero in
		$H^{\tilde{T}}_*(Y^{\{0,1\}})$ for dimension
		reasons
		since $(1,6,2)$ is not minimal in $W/W_{\{0,1\}}$.
		
	\end{example}
	
	\subsection{The unramified affine Springer fiber}
	
	\label{sec:unramgkm}

	The affine Springer fiber \cite{lusztig1991fixed}
	associated to an element $\gamma \in \Lie(G(\mathcal{K}))$ is
	\[\aff_{\gamma}^J=\left\{gI_J\in Y^J:
	g^{-1} \gamma g \in \Lie(I_J)\right\}.\] 
	We will be interested in the \emph{unramified}
	affine Springer fiber $Y_k\subset Y$ 
	of \cite{goresky2004unramified} associated to the element
	$\gamma_k=\diag(a_1z^k,...,a_nz^k)$, where 
	$a_i$ are distinct complex numbers and $k\geq 1$.
	We will also write $Y^\alpha_k$ to denote the corresponding
	parabolics $J=I_\alpha$ as for the flag variety.
	Then the action of the big torus $\tilde{T}$ on $Y$ preserves $Y_k$, and the fixed points are all of $W$.
	It was shown in \cite{goresky2003purity} 
	that $Y^J_{k}$ is paved by 
	affine spaces given by the intersections
	$Y^J_k \cap \Omega^J_{v W_J}$. We have the closures 
	$\overline{\Omega^J_{v W_J} \cap Y^J_k}$,
	which are strictly contained in the intersections 
	$Y^J_k\cap V^J_{v W_J}$. As above, we will drop the parahoric
	indices $J$ to denote the full flag version.
	
	For a finite lower set $P\subset W/W_J$, 
	the subvarieties $Y_k \cap V_P$ are GKM with respect
	to $\tilde{T}$, and we now describe the GKM graph, referring to \cite{goresky2004unramified}.
	An explicit description in the Grassmannian case
	for $GL_n$ was also explained
	in \cite{kivinen2020unramified}.
	The set of vertices is given by
	$V=\awgsl$, the entire Weyl group. We have
	\begin{equation}
		\label{eq:edgesasp}
		\outneighbors_k(v W_J)=\left\{u W_J \bruleq v W_J:
		u_-= t_{a,b}v_-=v_- t_{i,j},\ 
		|i-j|<kn\right\}.
	\end{equation} 
	where $u_-,v_-$ are the minimal representatives.
	The statistic $|i-j|$ does not
	depend on the representatives $\{i,j\}$ or their order,
	and is called the height of the 
	transposition. Then the GKM graph of $Y_k$
	is the one for which the outgoing neighbors
	of $v W_J$ are given by $\outneighbors_{k}(v W_J)$.
	Notice that we obtain the set
	of neighbors $\outneighbors(v W_J)$ for the affine flag
	variety as $k\rightarrow \infty$. 
	The leading terms $LT_k(v)$ is also
	determined by the weight function 
	\eqref{eq:affwf}.
	Just as in \eqref{eq:outneighborsflag},
	the parabolic versions are the same as the ones obtained
	by taking minimal representatives
	of the cosets $uW_J,vW_J$. We will similarly denote by 
	$\inv_{k}(v W_J)=|N^+_k(v W_J)|$ 
	the number of outgoing edges.

	The homology is defined as before as the direct limit
	\begin{equation}
		\label{eq:indhomologyk}
		H_*^{\tilde{T}}(Y^J_{k})=
		\lim_{\rightarrow} H_*^{\tilde{T}}(V^J_{P}\cap Y^J_{k}),
	\end{equation}
	over lower sets $P\subset W/W_J$,
	and we have the basis elements by 
	$A^J_{k,v W_J}=[V^J_{v W_J} \cap Y^J_{k}]\in H_*^{\tilde{T}}(Y^J_k)$.
	The natural map $H_*^{\tilde{T}}(Y_k)\rightarrow H_*^{\tilde{T}}(Y)$
	is injective, since they are both contained in the
	fixed point set, and is compatible with the left
	and right actions of $W$.
	
	As in the case of
	the flag variety, we can obtain the parabolic
	case by taking invariants for
	the Springer action:
	\begin{prop}\label{prop:invariants springer}
		The composition
		\begin{equation}\label{eq:composition springer}
			H_*^{\tilde{T}}(Y_k)^{W_J}\rightarrow H_*^{\tilde{T}}(Y_k)\rightarrow H_*^{\tilde{T}}(Y_k^J)
		\end{equation}
		is an isomorphism, where the leftmost term is the
		invariants with respect to the 
		right (i.e. Springer, star) action of $W_J$.
	\end{prop}
	\begin{rem}
		Results of this sort in the non-affine case
		can be found, for instance in \cite{borho1983partial}. 
		We give a direct proof using the explicit
		description in fixed points.
	\end{rem}
	
	\begin{proof}
		We have $H_*^{\tilde{T}}(Y_k)=H_*^{\tilde{T}}(Y_k)^{W_J}\oplus K$ where $K$ is the kernel of the multiplication on the right by $\sum_{w\in W_J} w$.
		
		Consider the map $\pi_J: H_*^{\tilde{T}}(Y_k)\rightarrow H_*^{\tilde{T}}(Y_k^J)$. On fixed points it is determined by the fact that it sends the fixed point corresponding to $w\in W_+$ for $w\in W$ to the fixed point corresponding to the class of $w$ in $W/W_J$. Thus we see that $\pi_J|_K=0$. Moreover, we see that the composition \eqref{eq:composition springer} becomes an isomorphism when tensored with $\tilde F$, so it is injective.
		
		It remains to show that $\pi_J$ is surjective. One way to see this is to observe that the projection induces an isomorphism between the cell corresponding to $w\in W/W_J$ and the cell corresponding to the minimal representative $w_-\in W$. This implies $\pi_J(A_{k,w_-}) = A_{k,w}^J$. Alternatively, we can observe that $\pi_J(A_{k,w_-})$ has the leading term of expected degree, and therefore $\pi_J(A_{k,w_-})$ for different $w$ generate the same submodule of $H_*^{\tilde T}(Y_k^J)\otimes_{\tilde S} \tilde F$ as $A_{k,w}^J$.
		
	\end{proof}
	
	\begin{example}
		For instance, we have an element
		\[A_{1,(1,6,2)}={\frac {p_{(3,4,2)}}{ \left( x_1-x_{{2}}+\eps \right)  \left( x_{{1}}
				-x_{{3}} \right) }}-{\frac {p_{(3,2,4)}}{ \left( x_1-x_{{2}}+\eps
				\right)  \left( x_{{1}}-x_{{3}} \right) }}+\]
		\begin{equation}
			\label{eq:ak162}{\frac {p_{(1,2,6)}}{
					\left(x_3 -x_{{2}}+\eps\right)  \left( x_{{1}}-x_{{3}} \right) }}-{
				\frac {p_{(1,6,2)}}{ \left(x_3-x_{{2}}+\eps\right)  \left( x_{{1}}-
					x_{{3}} \right) }}.
		\end{equation}
		Since we have an injective map $H_*^{\tilde{T}}(Y_k)
		\rightarrow H_*^{\tilde{T}}(Y)$, we must have that $A_{1,(1,6,2)}$
		expands positively in terms of the $A_{w}$ generators, and indeed
		using $\aaf$ we may check that $A_{1,(1,6,2)}=(x_2-x_3)A_{(1,6,2)}-A_{(1,5,3)}$.
	\end{example}
	
	\begin{rem}
		The equivariant cohomology may be defined as in Section
		\ref{sec:affgkm}, though now the inverse limit of graded algebras does not produce something 
		finitely generated, since there
		are infinitely many classes in each degree.
		This can be resolved using the alternative version,
		by considering the subalgebra of classes $\covar_w$
		which vanish on all but finitely many generators
		$A_{k,v}$. However, unlike in the case of the flag variety,
		this definition depends on the choice of generators
		$A_{k,v}$, whereas there are other potential candidates for the generators. For instance, using the intersection of closures
		$Y_{k} \cap V_w$ in place of the closures of the intersections
		would result in a different limit. 
		Defining the cohomology in the right way
		is interesting for the problem
		of generalizing Schubert positivity, but this is
		a red herring for this paper, 
		and so will not be defined. We will
		only need the equivariant cohomology in the usual sense for each 
		$Y_k\cap V_P$ individually.
	\end{rem}

	\subsection{The regular semisimple Hessenberg variety}
	
	\label{sec:gkmhess}
	
	We recall the GKM construction for the 
	regular semisimple Hessenberg variety.
	We will refer to \cite{demari1992hessenberg,abe2017cohomology,tymoczko2008permutation} for details.
	
	The regular semisimple Hessenberg variety
	is determined by a \emph{Hessenberg function}, 
	which in type A is equivalent to the data of a 
	Dyck path $\pi$, which in turn defines
	the \emph{unit interval order}
	$i<_\pi j$ if $i<j$ and $(i,j)\notin \uio(\pi)$.
	Let $M_{\pi}$ be the set of $n\times n$
	matrices $(a_{i,j})$ for which $a_{i,j}=0$
	whenever $i>_\pi j$.
	Let $P_\alpha \subset GL_n(\mathbb{C})$ be a parabolic
	subgroup for which the conjugation action
	preserves $M_\pi$, which for any $\pi$ 
	includes the upper triangular
	matrices $B$. Let
	$\gamma=\diag(a_1,...,a_n)$ be the diagonal
	matrix with all distinct complex entries.
	The regular semisimple Hessenberg variety
	is the subvariety of the usual
	complex partial flag variety given by
	\[\rsh^\alpha_{\pi}=\left\{gP_\alpha: g^{-1}\gamma g
	\in M_\pi\right\}\subset \Fl_\alpha.\]
	
	The composition $\alpha=\alpha(\pi)$ from Section
	\ref{sec:prev} determines a composition for which
	$P_\alpha\subset GL_n(\C)$ fixes $M_{\pi}$,
	and so we have a corresponding
	variety $\rsh_{\pi}^{\alpha'}$ whenever $P_{\alpha'}\subset P_\alpha$, in other words $\alpha'$ refines $\alpha$.
	In the example from Figure \ref{fig:dyckpathalpha},
	we would have
	\[M_{\pi}=\left\{
	\left(\begin{array}{cccccc}
		*&*&*&*&*&*\\
		*&*&*&*&*&*\\
		*&*&*&*&*&*\\
		0&0&*&*&*&*\\
		0&0&*&*&*&*\\
		0&0&0&*&*&*\end{array}\right)\right\},\quad
	P_{\alpha}=\left\{
	\left(\begin{array}{cccccc}
		*&*&*&*&*&*\\
		*&*&*&*&*&*\\
		0&0&*&*&*&*\\
		0&0&0&*&*&*\\
		0&0&0&*&*&*\\
		0&0&0&0&0&*
	\end{array}\right)\right\},\]
	and that $M_\pi$ is preserved
	by conjugation by $P_\alpha$, where $\alpha=(2,1,2,1)$.
	
	There is an action of the maximal $n$-dimensional torus $T\subset GL_n(\C)$, and $\rsh^{\alpha}_\pi$
	is GKM with respect to this action,
	with fixed points given by the entire symmetric
	group $V=S_n$. The edge 
	set is determined by
	\begin{equation}
		\label{eq:edgesrsh}
		\outneighbors_{\pi}(wS_\alpha)=\left\{vS_\alpha \bruleq w S_\alpha:
		v_-=w_- t_{i,j},(i,j) \in \uio(\pi)\right\}
	\end{equation}
	where $v_-,w_-$ are minimal coset representatives
	as above.
	We will denote the basis of homology by
	$A_{\pi,\sigma}\in H_T^*(\rsh_\pi)$.
	The generators of $H_T^*(\rsh^\alpha_\pi)$
	may again be determined as in \eqref{eq:affhat}.

	\begin{example}
		For instance, we would have
		\begin{equation}
			\label{eq:api312}
			A_{110100,(3,1,2)}=
			\frac{1}{x_1-x_3}p_{(3,1,2)}-
			\frac{1}{x_1-x_3}p_{(1,3,2)}=
			(x_2-x_3)A_{(3,1,2)}-A_{(2,1,3)},
		\end{equation}
		where $A_\sigma$ are the generators
		of the usual flag variety $H^*_T(\Fl_n)$.
	\end{example}

	We have the dot action of $S_n$ on the left
	on $H_*^T(\rsh_\pi^\alpha)$ and $H^*_T(\rsh_\pi)$ which is compatible
	with \eqref{eq:rshpd}.
	It was conjectured by Shareshian
	and Wachs, and then proved in two separate
	papers \cite{brosnan2018hessenberg,shareshian2016chromatic} that
	\begin{equation}\frob_Y
		\label{eq:shareshianwachs}
		H_T^*(\rsh_\pi)=(1-q)^{-n}\csfa_\pi[Y;q]=
		\omega\csf_{\dyckpath}[Y(1-q)^{-1};q]
	\end{equation}
	where $\csfa_{\pi}[Y;q]$ was defined
	in \eqref{eq:csfa}, and 
	\begin{equation}
		\label{eq:csf}
		\csf_{\dyckpath}[Y;q]=
		\sum_{\bn:(i,j)\in \uio(\dyckpath)\Rightarrow b_i\neq b_j} 
		q^{\inv_\dyckpath(\bn)}Y_{\bn}
	\end{equation}
	is Stanley's chromatic symmetric function.
	
	The equivariant/non-equivariant homology/cohomology are given as follows:
	\[
	H^*(\rsh_\pi) = (1-q)^{-n}\csfa_\pi[Y(1-q);q] = \omega\csf_{\dyckpath}[Y;q],\qquad H_*(\rsh_\pi) = \omega\csf_{\dyckpath}[Y;q^{-1}]
	\]
	\[
	H_*^T(\rsh_\pi) = \omega\csf_{\dyckpath}[Y(1-q)^{-1};q^{-1}] = (1-q)^{-n} \omega \csfa_\pi[Y;q^{-1}].
	\]
	Since $\rsh_\pi$ is smooth (though perhaps not connected)
	\cite{demari1992hessenberg}, we have an isomorphism
	of equivariant Borel-Moore homology with cohomology.
	For $\alpha=(1,...,1)$, it is given in the fixed point basis by
	\begin{equation}
		\label{eq:rshpd}
		H_*^T(\rsh_\pi)\rightarrow H_T^{*+\#\uio(\pi)},\quad
		p_\sigma\mapsto \left(\prod_{(i,j)\in \uio(\pi)} (x_{\sigma_i}-x_{\sigma_j})\right)p_\sigma.
	\end{equation}
	The multiplication map may be written as $f\mapsto f\Delta_\pi(\zn)$ where
	\[\Delta_\pi(\zn)=\prod_{(i,j) \in \uio(\pi)} (z_i-z_j),\]
	and $z_i$ is multiplication by the Chern class
	$p_\sigma z_i=x_{\sigma_i} p_\sigma$, which is natural
	to think of as right multiplication.
	Then Poincar\'e duality (on each connected component)
	is manifested via the identities
	\[
	\omega\csfa_\pi[Y;q^{-1}] = q^{-D(\pi)} \csfa_\pi[Y;q],\qquad \csf_{\dyckpath}[Y;q^{-1}] = q^{-D(\pi)} \csf_{\dyckpath}[Y;q].
	\]
	
	If $\alpha$ is some composition so that
	conjugation by $P_\alpha$ preserves $M_\pi$,
	then the fibers of
	the map $\rsh_\pi\rightarrow \rsh_\pi^\alpha$
	are isomorphic to the product of usual flag varieties
	$\Fl_{\alpha_1}\times \cdots \times \Fl_{\alpha_l}$,
	as the Hessenberg condition is trivial on the fibers.
	We therefore have that
	\[  \frob_Y H^*_T(\rsh_\pi^\alpha)=\csfa_{\pi,\alpha}[Y;q] 
	:= \frac{1}{(1-q)^n\aut_q(\alpha)} \csfa_\pi[Y;q],\]
	\begin{equation}
		\label{eq:authess}
		\frob_Y H^T_*(\rsh_\pi^\alpha)=
		(-q)^{-n}\omega_Y \csfa_{\pi,\alpha}[Y;q^{-1}]=
		\frac{q^{n'(\alpha)-\uio(\pi)}}{(1-q)^n \aut_q(\alpha)} \csfa_{\pi}[Y;q],
	\end{equation}
	where $n'(\alpha)$ was defined in \eqref{eq:macn}.
	
	\subsection{The Hessenberg paving of the affine Springer fiber}
	\label{sec:hesspav}
	We now describe an explicit
	presentation of the paving by affine bundles
	over the Hessenberg
	varieties defined in \cite{goresky2003purity}.
	We refer to that paper for all details 
	in this section.
	
	Let $P \subset \awg/S_\alpha$ 
	be a finite lower set,
	and suppose that $P=Q\cup S_n w S_\alpha$ for another
	lower set $Q$, where the union is disjiont,
	so that the double coset is maximal within
	$P$. Then we have the complement
	\[E^{\alpha}_{S_n w S_\alpha}=
	\Omega^\alpha_{S_n w S_{\alpha}} :=
	\bigsqcup_{v\in S_n w S_\alpha} \Omega^\alpha_{vS_\alpha}=V^\alpha_P-V^\alpha_Q\] 
	depends only on $S_n w S_\alpha$. 
	Moreover, it is isomorphic
	to an affine bundle of rank
	$\inv(w_-)$ over the parabolic (non affine) flag
	variety $\Fl_{\alpha'}=GL_n/P_{\alpha'}$, 
	where as above $w_+,w_-\in S_n w S_\alpha$
	are the maximal and minimal coset representatives.
	The composition $\alpha'$ is the 
	unique composition so that the parabolic subgroup
	$P_{\alpha'}$ is the stabilizer of $GL_n$ acting
	on $w_- P_\alpha$, which is the same as the composition
	$\alpha'=\alpha(S_n w S_\alpha)$ from Section \ref{sec:affperm}.
	Then we have that 
	$E^\alpha_{S_n w S_\alpha}\subset V^\alpha_{w_+}=
	G(\mathcal{O}) w I_{\alpha}$, and the map 
	$E^{\alpha}_{S_n w S_\alpha}\rightarrow 
	\Fl_{\alpha'}$ can be expressed by
	\[g(t) w I_\alpha\mapsto g(0)P_{\alpha'},
	\quad g(t)\in \mathcal{O}=I_{(n)}.\]
	
	By the results of \cite{goresky2003purity}, we have
	that $\aff_{k}^\alpha \cap E^{\alpha}_{S_n w S_\alpha}$
	is an affine sub-bundle over the Hessenberg variety
	$\rsh^{\alpha'}_{\pi} \subset \Fl_{\alpha'}$, 
	which form what is called a Hessenberg paving.
	The Dyck path is determined by 
	$\pi=\pi_k(S_nw S_\alpha)$, which is the path
	from Section \ref{sec:affperm}.
	It may be checked that 
	$\alpha(\pi_k(S_n wS_\alpha))$ refines
	$\alpha'=\alpha(S_n w S_\alpha)$, so 
	that the Hessenberg variety is well-defined.
	Geometrically, $\pi_k(S_n w S_\alpha)$ may be expressed as
	the unique path so that for $i<j$, we have
	\[(i,j)\in \uio(\pi)\Leftrightarrow 
	t_{a,b} w_{\wmin} \in \outneighbors_{k}(w_{-}),\]
	by comparing with \eqref{eq:wtodp}.
	
	\begin{example}
		\label{ex:hesspav}
		Consider the case of $n=2, k=1$
		shown in Figure \ref{fig:hesspav}.
		We see that 
		$\pi_1(S_2 (1,2) S_{1,1})=1100$
		so that $\rsh_{\pi}=\Fl_2$, and that the bundle
		has rank zero, giving a copy of $\mathbb{CP}^1$
		connecting $(1,2)$ to $(2,1)$. On the other hand,
		all other Dyck paths such as $\pi_1(S_2 (0,3) S_{1,1})$
		are equal to $1010$, for which the Hessenberg variety
		consists of two points, in this case 
		$\{(0,3),(-1,4)\}$.
		The bundle $E_{\{(0,3),(-1,4)\}}\cap Y_1$ has rank one,
		shown by incomplete line segments.
	\end{example}

	\begin{figure}
		\begin{centering}
			\begin{tikzpicture}
				\node[text width=.5cm] at (-1.5,.5) {...};
				\draw[-,thick,color=black] (-1,0)--(0-.2,1-.2);
				\filldraw[black] (0,1) circle (2pt) 
				node[anchor=south]{$(-2,5)$};
				\draw[-,thick,color=black] (0,1)--(1-.2,0+.2);
				\filldraw[black] (1,0) circle (2pt) 
				node[anchor=north]{$(3,0)$};
				\draw[-,thick,color=black] (1,0)--(2-.2,1-.2);
				\filldraw[black] (2,1) circle (2pt) 
				node[anchor=south]{$(0,3)$};
				\draw[-,thick,color=black] (2,1)--(3-.2,0+.2);
				\filldraw[black] (3,0) circle (2pt) 
				node[anchor=north]{$(1,2)$};
				\draw[-,thick,color=black] (3,0)--(4,1);
				\filldraw[black] (4,1) circle (2pt) 
				node[anchor=south]{$(2,1)$};
				\draw[-,thick,color=black] (4+.2,1-.2)--(5,0);
				\filldraw[black] (5,0) circle (2pt) 
				node[anchor=north]{$(-1,4)$};
				\draw[-,thick,color=black] (5+.2,0+.2)--(6,1);
				\filldraw[black] (6,1) circle (2pt) 
				node[anchor=south]{$(4,-1)$};
				\draw[-,thick,color=black] (6+.2,1-.2)--(7,0);
				\filldraw[black] (7,0) circle (2pt) 
				node[anchor=north]{$(-3,6)$};
				\draw[-,thick,color=black] (7+.2,0+.2)--(8,1);
				\node[text width=.5cm] at (8.5,.5) {...};
			\end{tikzpicture}
			\caption{An illustration of the Hessenberg paving
				of $Y_1$ for $n=2$. The fibers of the bundle
				$E_{S_2 w S_{1,1}}$ are shown using incomplete
				segments.}
			\label{fig:hesspav}
		\end{centering}
	\end{figure}
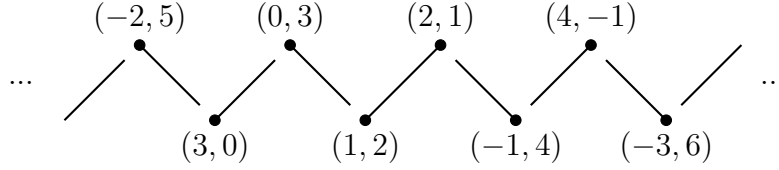

	We may describe the equivariant
	Euler class of the fiber of $E^\alpha_{S_n w S_\alpha}$
	over a fixed point:
	\begin{equation}
		\label{eq:hesspaveeuler}
		e\left(E^{\alpha}_{S_n wS_{\alpha}} \big|_{\sigma w_-S_{\alpha}}\right)=
		\frac{ \lt_{k}(\sigma w_- S_\alpha)}{\lt_\pi(\sigma S_{\alpha'})}
		\in \tilde{S}.
	\end{equation}
	Note that we obtain a different formula for the
	rank of $E^\alpha_{S_n w S_\alpha} \cap Y_k$ over each point, in particular
	\begin{equation}
		\label{eq:rankealpha}
		\rank(E^\alpha_{S_n w S_\alpha} \cap Y_k)=
		\inv_k(w_- S_\alpha)=
		\inv_k(w_+ S_\alpha)-\inv_\pi(\sigma_0 S_{\alpha'})
	\end{equation}
	where $w_{\pm}$ are the maximal and minimal
	elements of $S_n w S_\alpha$ respectively, and
	$\sigma_0=(n,...,1)$ is the maximal element of $S_n$.

	We now make the following
	observation: if $U=X-Z$ and all three spaces
	are paved by affines, then we have that
	the long exact sequence
	\begin{equation} 
		\label{eq:longshort}
		\cdots \rightarrow H_*^T(Z)\rightarrow H_*^T(X)\rightarrow H_*^T(U)\rightarrow H_{*-1}^T(Z)\rightarrow \cdots
	\end{equation}
	is actually short exact, as the Borel-Moore homology
	is concentrated in even degree, which in particular
	implies additivity of characters.

	Applying this in our situation, 
	the projection map induces an isomorphism
	\[H_*^{\tilde{T}}(E^{\alpha}_{S_n wS_\alpha}\cap 
	\aff^{\alpha}_{k})\cong H_*^{\tilde{T}}(\rsh_{\pi}^{\alpha'})
	\cong \tilde{S} \otimes_S H_*^{T}
	(\rsh_{\pi}^{\alpha'}),\] 
	where the second map exists since
	loop rotation only acts on the fibers of the bundle.
	We then have a short exact sequence
	\begin{equation}
		\label{eq:hessiso}
		0\rightarrow H_*^{\tilde{T}}(V^\alpha_Q \cap Y^{\alpha}_{k})
		\rightarrow H_*^{\tilde{T}}(V^\alpha_P\cap \aff^{\alpha}_{k})
		\rightarrow H_*^{\tilde{T}}(\rsh_{\pi}^{\alpha'})\rightarrow 0 \end{equation}
	since all terms are nonzero in only even degree.

	The following proposition summarizes these statements, 
	and follows by interpreting the results of
	\cite{goresky2003purity} in type A.
	\begin{prop}
		\label{prop:asptohess}   
		Let $E^{\alpha}_{S_n w S_\alpha}=V^\alpha_{P}-V^\alpha_Q$, 
		$P=Q\cup S_n w S_{\alpha}\subset W/S_{\alpha}$, 
		$\alpha'=\alpha(S_n w S_{\alpha})$, 
		$w_{\pm} \in S_n w S_\alpha$
		and $\pi=\pi_{k}(S_n wS_\alpha)$ be as above.
		Then $\yk^{\alpha}\cap E^\alpha_{S_n wS_\alpha}$ is an affine bundle of rank $\inv_k(w_-)$ over the Hessenberg variety
		$\rsh_{\pi}^{\alpha'}$.
		Moreover, the Schubert cell
		$\Omega_{\sigma w_-S_\alpha} \cap \yk^\alpha \cap E^\alpha_{S_n w S_\alpha}$ is
		identified with the restriction
		of this bundle to $\Omega_{\sigma S_{\alpha'}} \cap \rsh_{\pi}$,
		and the surjection in \eqref{eq:hessiso}
		is given in the fixed point basis by
		\begin{equation}
			\label{eq:asptohess}
			\hessmap^\alpha_{S_n w S_\alpha}:
			p_{\sigma w_- S_\alpha} \mapsto
			\frac{\lt_{k}(\sigma w_- S_\alpha)}
			{\lt_\pi(\sigma S_{\alpha'})}
			p_{\sigma S_{\alpha'}}
		\end{equation}
		with all other fixed points $p_{v S_{\alpha}}$ for
		$v S_{\alpha}\notin Q$ mapping to zero, which in
		particular satisfies
		$h^\alpha_{S_n w S_\alpha}(A^{\alpha}_{k,\sigma w_- S_\alpha})=A^{\alpha'}_{\pi,\sigma S_{\alpha'}}$.
	\end{prop}

	\begin{example}
		If $w=(1,6,2)$, then 
		\[S_3w=\left\{
		(1, 5, 3), (1, 6, 2), (2, 4, 3), (2, 6, 1), (3, 4, 2), (3, 5, 1)\right\}\]
		and $w_{-},w_{+}=(2,4,3),(2,6,1)$.
		We find that $\pi=\pi_1(w_-)=110100$,
		$\uio(\pi)=\{(1,2),(2,3)\}$,
		and also that $\lt_{3}(w_{\wmin})=x_1-x_3+\eps$.
		Inserting this into \eqref{eq:asptohess}, 
		and plugging in the formulas from
		\eqref{eq:ak162} and \eqref{eq:api312},
		we verify that
		\[h_{(2,4,3)}(A_{1,(1,6,2)})=
		A_{110100,(3,1,2)},\]
		noticing that the $\epsilon$ variables cancel.
	\end{example}

	\section{GKM spaces and the nabla operator}
	\label{sec:thmb}
	
	In this section we state and prove Theroem A
	from the introduction,
	which connects the results of \cite{carlsson2020combinatorial}
	with the nil Hecke algebra and the homology of $Y_{k}$.
	
	\subsection{Connection with the nil Hecke algebra}
	
	\label{sec:nablahecke}
	
	Consider the ind-subvariety 
	\begin{equation}
		\label{eq:zpdef}
		Z =\left\{V_0\subset \cdots \subset V_n \in Y: 
		\mbox{$\C[[z]] e_i \subset V_n$ for all $i$}\right\}. 
	\end{equation}
	Then the torus fixed points $Z$ are the subset $W_+\subset W$,
	which is in fact a lower set in the Bruhat order,
	and since it is preserved by $I$, it is a union of Schubert varieties. We have its equivariant homology over the small torus
	\[H_*^T(Z) = H_*^T(Y) \cap F\cdot W_+\]
	where $F=\C(\xn)$ be the field of fractions of $S$,
	and the intersection is taken in $F\cdot W$.
	Then from Proposition \ref{prop:aaf} of Section
	\ref{sec:nilhecke},
	we have the explicit description
	\[H_*^T(Z) = \bigoplus_{w \in W_+} S\cdot A_w
	\subset \bigoplus_{d\geq 0} \aaf \rho^d\]
	where $\aaf$ is the nil Hecke algebra over $S$.
	
	Let $M_k=H_*^T(Z_k)$ as an $S$-submodule of 
	$H_*^T(Z)$. We have that $M_k$ is bigraded by the degree
	in the $x$-variables, and also by the positive
	grading $|w|$ for $w\in W_+$. It is also preserved by the 
	left and right actions of $S_n$, which preserve
	the condition in \eqref{eq:zpdef}.
	We may therefore consider the Frobenius character
	\[\frob_{Y,X} M_k \in \Sym_{q,t}(X,Y).\]
	The left and right actions correspond to 
	to the variables $Y,X$ respectively.
	In other words the symmetric function in 
	the $Y$-variables keeps track of the dot action,
	while $X$ corresponds to the choice of parahoric
	subgroup which must contain the zero label $0\in J$,
	in other words come from a composition.
	We now show that $F_{Y,X} M_k$ is determined
	by the formula from Theorem \ref{thm:oldthm}.

	\begin{thm}
		\label{thm:frob}
		Let $\Delta(\xn)=\prod_{i<j} (x_i-x_j)$ be
		the alternating element.
		\begin{enumerate}[(a)]
			\item \label{thm:intitem} We have
			\[
			H_*^T(Y_k)=H_*^T(Y) \cap \Delta(\xn)^{-k} S\cdot W
			\]
			where the intersection is taken
			as $S$-submodules of $F\cdot W$.
			\item \label{thm:resitem} The restriction of the surjection
			$H_*^T(Y_k)\rightarrow H_*^T(Y^{J}_k)$ to the invariants
			$H_*^T(Y_k)^{W_J}$ is an isomorphism.
			\item \label{thm:frobitem} The Frobenius character
			is given by
			\[\frob_{Y,X} H_*^T(Z_k)=q^{-k\binom{n}{2}}\omega_X 
			\nabla_X^k e_n\left[\frac{XY}{(1-q)(1-t)}\right].\]
		\end{enumerate}
	\end{thm}
	
	\begin{example}
		We have
		\[[X^{(1,1)} Y^{(1,1)} ]
		q^{-1} \omega_X \nabla_Xe_2\left[\frac{XY}{(1-q)(1-t)}\right] = \frac{1+q^{-1}}{(1-q)^2}+
		t\frac{1+3q^{-1}}{(1-q)^2}+
		t^2\frac{1+5q^{-1}}{(1-q)^2}+\cdots\]
		Each coefficient of $t^d$ 
		consists of a shifted copy of part the space
		shown with paving in Example \ref{ex:hesspav},
		from taking only those fixed points in $W_+$
		which have total degree $|w|=d$.
	\end{example}
	
	We need a lemma.
	
	\begin{lem}
		\label{lemma:maxdenom}
		Let $T=(z_1,...,z_n)$ be the small torus, with
		no loop rotation. Then for any element
		\begin{equation}
			\label{eq:denomlem}
			\sum_{w}c_w(\xn) p_w\in H_*^T(\yk)\subset
			F\cdot W,
		\end{equation}
		we have that $\Delta(\xn)^k c_w(\xn)\in S$.
	\end{lem}
	
	\begin{proof}
		It suffices to prove that for any $v\bruleq w$,
		and any $1\leq i < j \leq n$,
		there is a class 
		$\chi_{v}\in H^*_T(\yk \cap V_{w})$
		which is supported at $v$, and which only vanishes
		to degree $k$ at $x_i=x_j$.
		It thus follows that
		$f(\xn)(x_i-x_j)^k c_v=\langle A,\chi_v\rangle\in S$
		where $A$ is the element from \eqref{eq:denomlem},
		and $f(\xn)$ is not divisible by $x_i-x_j$,
		so that the order with which $x_i-x_j$ enters the denominator of $c_v$ is at most $k$.
		
		First, consider the case of $n=2$, and
		call that Weyl group $\awgsl'$.
		Then for any $w$, we have that
		\[\chi_v(w)=\begin{cases} (x_1-x_2)^k & w=v\\
			0 & \mbox{otherwise}\end{cases}\]
		is an element of $H^*_T(\yk \cap V_w)$.
		This is seen by noticing that
		the leading terms in
		$\covar_w$ divide into $(x_1-x_2)^k$, and that
		degree coincides with degree of vanishing of $x_1-x_2$,
		so that $\chi_v$ is related to $\covar_w$ by an 
		upper-triangular
		change of basis matrix with coefficients in $S$, and ones on
		the diagonal.
		
		Now suppose $n>2$. For any $1\leq i<j\leq n$,
		we have the subgroup $\awgsl_{i,j}\subset \awgsl$
		of all $w$ for which $w_{k}=k$ provided $k\notin \{i,j\}$, which is generated
		by $t_{i,j},t_{j-n,i+n}$, and is isomorphic to the
		affine Weyl group for $n=2$. Now, for any $v$,
		consider the isomorphism of $W$-sets
		which sends the minimal element of $W_{i,j} v$ 
		in the Bruhat order to the identity in $W'$,
		and let $w'\in W'$ be the image of $w$
		for any $w\in W_{i,j}v$.
		Then we have an element 
		$\tilde{\chi}_v(w)\in \tilde{S}$
		\[\tilde{\chi}_{v}(w)=\begin{cases}
			\tilde{\chi}'_{v'}(w')\big|_{x_1=x_i,x_2=x_j} & w\in \awgsl_{i,j} v\\
			0 & \mbox{otherwise}
		\end{cases}\]
		where $\tilde{\chi}_v' \in H^*_{\tilde{T}}(\yk')$
		is any element that lifts $\chi_{v'}$ from the last paragraph
		to the big torus $\tilde{T}'$. Then using the GKM
		relations, we see that
		\[\left(\prod_{u\in \left(\inneighbors(v) \cup \outneighbors(v)\right) -\awgsl_{i,j}}
		\wf(u,v)\right) \tilde{\chi}_{v} \in 
		H^*_{\tilde{T}}(Y),\]
		and maps to the desired element in
		$H^*_T(Y_k)$ under the specialization
		$\tilde{S}\rightarrow S$.
		
	\end{proof}
	We can now prove the theorem.

	\begin{proof}
		
		Denote by $M$ the module on the right side in \ref{thm:intitem}. Then the first inclusion
		$H_*^T(Y_k)\subset M$ follows from
		Lemma \ref{lemma:maxdenom}.
		
		To check the reverse inclusion, we compare the
		leading terms. Let $P\subset W$ be any 
		lower set, let $w \in P$ be any maximal element,
		and suppose $Q=P-\{w\}$.
		Then we have a map
		$M \cap H_*^T(V_P)/M \cap H_*^T(V_Q)\rightarrow S$
		which extracts the coefficient of $p_w$
		and multiplies by $\Delta(\xn)^k$,
		and similarly for $H_*^T(Y_k)$, whose images
		are given by ideals $I,I'\subset S$ respectively. 
		Then $I'$ is principally generated
		by $f_k(w)=\Delta(\xn)^k \lt_{k}(w)^{-1}\in S$.
		On the other hand, using the explicit form
		of \eqref{eq:edgesasp}, we have
		\begin{equation}
			\label{eq:dinvgcd}
			f_k(w)=\prod_{1\leq i<j\leq n} (x_i-x_j)^{\max
				\left(k-\#\outneighbors_{i,j}(w),0\right)},
		\end{equation}
		where $\outneighbors_{i,j}(w)$ is the set
		of elements $t_{a,b}w\in \outneighbors(w)$
		for the affine flag variety, i.e. inversions
		of $w^{-1}$,
		for which $a,b$ are congruent to 
		$i,j$ modulo $n$. But by the definition of
		$M$, each element of $I$ must be
		divisible by this element, so we have $I=I'$.
		By induction on the lower set $P$, we have
		that $H_*^T(Y_k)=M$. Notice also that
		$\deg\left(f_k(w)\right)=
		\dinv_k(w)$, using \eqref{eq:dinvw}.
		
		We have already established \ref{thm:resitem} in Proposition \ref{prop:invariants springer}.
		
		To prove \ref{thm:frobitem}, we have
		\[[X^\alpha] \frob_{Y,X} H_*^T(Z_k)= \frob_{Y}
		H_*^T(Z_k^{\alpha})\]
		\[=\sum_{S_n w S_\alpha \in S_n \backslash W_+ /S_\alpha}
		t^{|w|}q^{-\rank\left(E^\alpha_{S_n w S_\alpha} \cap Z^\alpha_k\right)}
		\frob_Y H_*^T(\rsh^{\alpha'}_{\pi})\]
		\[=\sum_{S_n w S_\alpha \in S_n \backslash W_+ /S_\alpha}
		t^{|w|}q^{\#\uio(\pi)-n'(\alpha')-\inv_k(w_+ S_\alpha)}
		\frob_Y H_*^T(\rsh^{\alpha'}_{\pi})\]
		\[=\sum_{S_n w S_\alpha \in S_n \backslash W_+ /S_\alpha}
		\frac{t^{|w|}q^{-\inv_k(w_+S_\alpha)}}
		{(1-q)^n\aut_q(\alpha')}
		\csfa_{\pi}[Y;q]\]
		\[=[X^\alpha]\sum_{[\mn,\an]}
		\frac{t^{|\mn|}q^{\dinv_k(\mn,-\an)-k\binom{n}{2}}}
		{(1-q)^n\aut_q(\mn,\an)} X_\an
		\csfa_{\pi}[Y;q]\]
		\begin{equation}
			\label{eq:frobspring}
			=[X^\alpha]
			\omega_X q^{-k\binom{n}{2}} \nabla^k_X 
			e_n\left[\frac{XY}{(1-q)(1-t)}\right].
		\end{equation}
		
		In the first line we used item \ref{thm:resitem}
		to compute the character in terms of the parabolic.
		To get to the second line, we used the additivity of Frobenius characters from \eqref{eq:longshort}.
		In the third line we took the second formula
		for the rank of $E^\alpha_{S_n w S_\alpha}\cap Y_k$ from 
		\eqref{eq:rankealpha}, where $w_+\in S_n w S_\alpha$ is a maximal element, remembering that 
		$\inv_k(w_+ S_\alpha)=\inv_k(w_-)$ when $w_-$ is a minimal
		element in $w_+ S_\alpha$. To get the next line,
		we apply \eqref{eq:authess}, which cancels
		some of the exponents of $q$ in the numerator.
		In the second to last line, we apply the
		rules at the end of Section \ref{sec:affperm},
		and in the final line we apply Corollary
		\ref{cor:supthm}.
		
	\end{proof}

	\section{The lattice action and nabla positivity}
	
	We now analyze the action of the lattice $\mathbb{Z}^n$
	on $H_*^T(Z_k)$, and present a new 
	conjecture which would categorify
	(and therefore imply) nabla positivity,
	Conjecture \ref{conj:nabpos}.
	
	\subsection{The lattice action}
	\label{sec:koszul}
	
	We have a lattice $\mathbb{Z}^n \subset W$ 
	generated by the elements
	$y_i:i\mapsto i+n$, and leaves $j$ fixed
	when $j$ is not congruent to $i$ modulo $n$, so that 
	$W=\mathbb{Z}^n\rtimes S_n=S_n \ltimes\mathbb{Z}^n$.
	The full Weyl group $W$ acts on the 
	left by the dot action on homology,
	though not on the space level, since the $S_n$-component
	permutes the diagonal elements of $\gamma_k$.
	However, the lattice acts on the space level
	since the translation elements preserve the
	diagonal matrix $\gamma_k$, and in
	fact the quotient $\mathbb{Z}^n\backslash Y_{k}$ is a
	central object from \cite{goresky2004unramified},
	in connection with orbital integrals in general
	root systems. They studied the homology group 
	$H_*(\Z^n \backslash \yk)$ 
	from induced action
	of $\C[y_1^{\pm 1},...,y_n^{\pm 1}]=\C[\Z^n]$  on $H_*(\yk,\mathcal{L})$ with local coefficients. 
	
	The positive affine Springer
	fiber $\zk$ is only invariant under the semigroup
	$\mathbb{Z}_{\geq 0}^n \subset \mathbb{Z}^n$
	generated by $y_i$. This corresponds to a
	restricted action of the polynomial algebra
	$\C[y_1,...,y_n]\subset
	\C[y_1^{\pm 1},...,y_n^{\pm 1}]$ on 
	$H^T_*(\zk)$, in which the degree in 
	$y_i$ is the $t$ grading in nabla formulas.
	We will show directly that this action is free,
	which has an analogue on the polygraph or
	``coherent'' side mentioned in the introduction.
	
	The space $F \cdot W$ is naturally isomorphic to the algebra $F[y_1^{\pm 1},...,y_n^{\pm 1}] \rtimes S_n$. We have
	\begin{prop}
		We have
		\begin{enumerate} 
			\item \label{item:algprop1} The image of $H^T_*(Y)$ in $F[y_1^{\pm 1},...,y_n^{\pm 1}]\rtimes S_n$ is the algebra generated by the finite difference operators $A_i=\frac{1-s_i}{x_i-x_{i+1}}$ ($i=1,\ldots,n-1$) and the elements $x_i, y_i^{\pm 1}$ ($i=1,\ldots,n$).
			\item \label{item:algprop2} The image of $H^T_*(Z)$ in $F[y_1,...,y_n] \rtimes S_n$ is the algebra generated by  $A_i$ ($i=1,\ldots,n-1$) and the elements $x_i, y_i$ ($i=1,\ldots,n$).
		\end{enumerate}
	\end{prop}
	\begin{proof}
		The operator $A_n$ can be written in terms of $\rho$ as
		\[
		A_n = \rho A_{n-1} \rho^{-1},
		\]
		whereas $\rho$ can be obtained from 
		\[
		\rho = y_1 s_1 \cdots s_{n-1},\qquad s_i = 1 - (x_i - x_{i+1}) A_i,
		\]
		Proving the first part. To show the second, 
		consider $A_w$ for any $w\in W_+$. If $w=w' s_i$ with $w'<w$ and $1\leq i \leq n-1$ we can write $A_w=A_{w'} A_i$, otherwise we would have $0<w_1<\ldots< w_n$. If $w$ is not the identity permutation, then $w_n>n$ and we have $w'=w \rho^{-1}\in W_+$, so we can write $A_w=A_{w'} \rho$.
	\end{proof}
	
	\subsection{The symmetry}
	
	Recall that the subspace $H^T_*(Z_k)\subset H^T_*(Z)$ is described as $H^T_*(Z_k)=M_k$, where $M_k=H_*^T(Z) \cap \Delta(\xn)^{-k}  S[y_1,\ldots,y_n] \rtimes S_n$. The Frobenius character of $H_*^T(Y_k)$ as given by 
	Theorem \ref{thm:frob}
	when multiplied by $q^{k\binom{n}2}$ is a symmetric function in $q$ and $t$. Thus one may guess that there is some sort of symmetry that interchanges the $x$ and the $y$ variables.
	\begin{thm}
		\label{thm:gl2}
		There is an action of $GL_2(\C)$ on $M_k$ such that the diagonal torus action corresponds to the bigrading and the involution $\begin{pmatrix}0&1\\1&0\end{pmatrix}$ satisfies
		\[
		\begin{pmatrix}0&1\\1&0\end{pmatrix} x_i \begin{pmatrix}0&1\\1&0\end{pmatrix} = y_i.
		\]
	\end{thm}
	\begin{proof}
		Consider the differential operators
		\[
		E = \sum_{i=1}^n x_i \frac{\partial}{\partial y_i}, \quad F = \sum_{i=1}^n y_i \frac{\partial}{\partial x_i}, \quad H = \sum_{i=1}^n x_i \frac{\partial}{\partial x_i} - \sum_{i=1}^n y_i \frac{\partial}{\partial y_i}.
		\]
		Let these operators act on $F[y_1,\ldots,y_n]$ and therefore on $F[y_1,\ldots,y_n] \rtimes S_n$ coefficient-wise. It is straightforward to check that the triple $(E,F,H)$ is a representation of the Lie algebra $sl_2$, and that these operators preserve $S[y_1,\ldots,y_n] \rtimes S_n$. Let us verify that they also preserve $\Delta(\xn)^k H_*^T(Z)$. We easily check that
		\[
		[E,A_i] = 0,\qquad [E,x_i]=0,\qquad [E,y_i]=x_i,
		\]
		\[
		[H,A_i] = -A_i,\qquad [H,x_i] = x_i, \qquad [H,y_i]=-y_i,
		\]
		\[
		[F,x_i]=y_i,\qquad [F,y_i] = 0,
		\]
		and with a little bit more work that
		\[
		[F,A_i] = - \frac{y_i - y_{i+1}}{x_i-x_{i+1}} A_i, \qquad   \frac{y_i - y_{i+1}}{x_i-x_{i+1}} =A_i y_i -  y_{i+1} A_i.
		\]
		Moreover, we have
		\[
		E \Delta(\xn)^k = \Delta(\xn)^k E,\qquad H \Delta(\xn)^k = \Delta(\xn)^k (H+k\tbinom{n}2),
		\]
		\[
		F \Delta(\xn)^k = \Delta(\xn)^k \left(F + k \sum_{i<j} \frac{y_i - y_{j}}{x_i-x_{j}}\right).
		\]
		Thus we see that the operators $E,F,H$ preserve both $\Delta(\xn)^k H_*^T(Z)$ and $S[y_1,\ldots,y_n] \rtimes S_n$, and hence also the intersection $\Delta(\xn)^k M_k$.
		
		Since the bigrading on $\Delta(\xn)^k M_k$ has only non-negative degrees, the operators $E$, $F$ are locally nilpotent, which implies that the action of $sl_2$ extends to an action of $SL_2(\C)$. The center of $GL_2(\C)$ acts by the total degree, giving an action of $GL_2(\C)$.
		
		Now let us verify the involution property. 
		We begin by writing
		\[
		\begin{pmatrix}0&1\\1&0\end{pmatrix} = \begin{pmatrix}1&0\\0&-1\end{pmatrix} \exp(E) \exp(-F) \exp(E).
		\]
		Then the statement follows from the following computations:
		\[
		\exp(E) x_i = x_i \exp(E),\qquad \exp(-F) x_i = (x_i-y_i) \exp(-F),
		\]
		\[
		\exp(E)(x_i-y_i) = -y_i \exp(E),\qquad \begin{pmatrix}1&0\\0&-1\end{pmatrix} y_i = -y_i \begin{pmatrix}1&0\\0&-1\end{pmatrix}.
		\]
	\end{proof}
	Since $M_k$ is free over $S$ we obtain
	\begin{cor}
		\label{cor:yaction}
		We have that $H_*^{T}(\ypk)$ 
		is free as a module over $\C[\yn]$.
	\end{cor}
	
	\begin{cor}
		\label{cor:frobquot}
		We have that
		\begin{equation}
			\label{eq:frobquot}
			\frob_{Y,X} (N_k)=q^{-k\binom{n}{2}}\omega_X
			\nabla^k e_n\left[\frac{XY}{1-q}\right].
		\end{equation}
		where $N_k=H_*^{T}(\ypk)/(y_1,...,y_n)H_*^{T}(\ypk)$.
	\end{cor}
	
	\begin{proof}
		This follows by applying the plethystic substitution 
		$Y\mapsto Y(1-t)$ to both sides of 
		Theorem \ref{thm:frob}, item \ref{thm:frobitem}.
	\end{proof}
	
	\subsection{Koszul submodules}
	We begin by describing the $2^n$ submodules
	\[\yn^{\sn} H^T_*(\zk)=(y_1^{s_1}\cdots y_n^{s_n}) H^T_*(\zk)\]
	for $s_i\in \{0,1\}$, which constitute the terms
	in the Koszul resolution,
	as the homologies of certain closed 
	ind-subvarieties $Z^{\sn}_k\subset Z_k$ as follows.
	\begin{defn}
		For each $\sn=(s_1,...,s_n)\in \{0,1\}^n$, define
		\[Z^{\sn}=\left\{V_0\subset \cdots \subset V_n
		\in \aff: z^{-s_i}\C[[z]] e_i \in V_n\right\}\]
		considered as an ind-variety 
		\[Z^\sn=\lim_{\rightarrow} Z^{\sn} \cap V_w\]
		over all Schubert varieties $V_w\subset \yp$.
		Let $Z^{\sn}_k=Z^{\sn} \cap \ypk$ be the intersection with the affine Springer fiber, also as an ind-variety,
		and let $Z^i_k=Z^{\sn}_k$ where $\sn$ 
		is defined by $s_i=1$, and all other entries are zero.
		In particular, we have $Z_k^{0^n}=\zk$, and the fixed point
		set is given by $\left(Z^{\sn}_k\right)^T=W_{\sn}$ where
		\begin{equation}\label{eq:fixedzk}
			W_{\sn}= \yn^{\sn} W_+=
			\left\{
			w\in W_+: w^{-1}_i/n\leq 1-s_i\mbox{ for $1\leq i \leq n$}\right\}.
		\end{equation}
	\end{defn}
	
	We have the following proposition.
	\begin{prop}
		\label{prop:zpsinj}
		Let $P\subset W_+$ be a lower set which is invariant
		under the left $S_n$ action, and let $V_P \subset \yp$
		be the corresponding union of Schubert varieties. Then we have
		\begin{enumerate}
			\item \label{prop:zpsinj1} Each $Z^{\sn}_k\cap V_P$ is a GKM subvariety of $\zk$ with
			respect to the big torus $\tilde{T}$ and we have $H_*^{T}(Z^{\sn}_k \cap V_P) = H_*^{T}(Z^{\sn}_k) \cap H_*^{T}(V_P)$.
			\item \label{prop:zpsinj2} If $Q \subset P$ and $\tn\geq \sn$ componentwise, then
			$H_*^{T}(Z^{\tn}_k\cap V_Q)\hookrightarrow H_*^{T}(Z^{\sn}_k \cap V_P)$
			is injective, and the image 
			splits as a direct summand of $S$-modules.
			\item \label{prop:zpsinj3} The image of $H_*^{T}(Z^{\sn}_k)\hookrightarrow H_*^{T}(\ypk)$
			is given by $\yn^\sn H_*^T(\ypk)$. 
		\end{enumerate}
	\end{prop}
	
	\begin{proof}
		To prove the GKM statement, it suffices to 
		assume $\sn=(1^l0^{n-l})$ with all ones on the left,
		using the left $S_n$-symmetry. Then $Z_k^{\sn}$
		is a union of Schubert varieties, as it
		is isomorphic to $\rho^k Z_k$ via the rotation 
		operator $\rho$ on $\yk$, which preserves the
		Bruhat order. In other words, each of the $Z^{\sn}_k$
		is a union of Schubert varieties rotated
		by some element of $S_n$. Hence $Z_k^{\sn}\cap V_P$ is also a union of Schubert varieties and therefore is GKM. The homology of $Z_k^{\sn}\cap V_P$ is spanned by the classes of cells which belong both to $Z_k^{\sn}$ and $V_P$, and since the classes of cells form a free basis of $H^T_*(Z_k)$ we see that $H_*^{T}(Z^{\sn}_k \cap V_P) = H_*^{T}(Z^{\sn}_k) \cap H_*^{T}(V_P)$.
		
		For the second statement we may assume that
		$Q\subset P$, $t_i\geq s_i$, and 
		$\sn,\tn$ are both sorted so that 
		the ones are all on the
		left, using the $S_n$ action simultaneously. 
		In this case both $Z^{\sn}_k\cap V_P,Z^{\tn}_k\cap V_Q$ are
		unions of Schubert varieties, which determines the
		splitting.
		
		For the final statement, note that both homologies embed into $H^T_*(Y_k)$ on which $\yn^\sn$ acts as an automorphism. The automorphism comes from the geometric automorphism which sends $Z_k$ to $Z_k^\sn$. Thus $\yn^\sn$ bijectively maps $H^T_*(Z_k)$ to $H^T_*(Z_k^\sn)$.

		
	\end{proof}

	Let $\mathcal{Z}_{k}=\left\{Z^{\sn}_k \cap V_P 
	\subset \zk \right\}$
	be the collection of all subspaces appearing in the above proposition, for
	\emph{finite} lower sets $P\subset W_+$,
	which is itself a poset by inclusion, and is closed under taking
	intersections. By the above proposition, the homology $H^T_*(Z^{\sn}_k \cap V_P)\subset H^T_*(Z_k)$ is a free $S$-module.

	\begin{prop}
		\label{prop:dist-lat}
		The submodules of the form $H^T_*(Z^{\sn}_k \cap V_P) \subset H^T_*(Z_k)$ generate a distributive lattice. For any union of elements of $\mathcal{Z}_{k}$ we have
		\[
		H^T_*((Z^{\sn_1}_k \cap V_{P_1}) \cup \cdots \cup (Z^{\sn_m}_k \cap V_{P_m})) = \]
		\[(H^T_*(Z^{\sn_1}_k) \cap H^T_*(V_{P_1})) + \cdots + (H^T_*(Z^{\sn_m}_k) \cap H^T_*(V_{P_m})).
		\]
	\end{prop}
	\begin{proof}
		Note that the two statements are essentially equivalent (see Proposition \ref{prop:homology 1}). 
		In fact we will prove the statement for more general subsets of the form $Z_k^\sn\cap V_P$, where $P\subset W_+$ is a lower set in the Bruhat order, 
		which is not necessarily $S_n$-invariant.
		
		The set $Z_k$ is paved in two ways: in one way by Bruhat cells, in another way by the $2^n$ subsets of the form
		\[
		Z_k^{\sn\, o}:= Z_k^\sn \setminus \bigcup_{\substack{\tn\geq\sn\\ \tn\neq\sn}} Z_k^\tn.
		\]
		Consider the paving obtained by intersecting these two pavings. A cell in the new paving is the intersection of the Schubert cell $I w I/I\cap Y_k$ corresponding to some $w\in W_+$ with the set $Z_k^{\sn\, o}$ for some $\sn\in\{0,1\}^n$. To complete the proof it sufficient to show that the equivariant homology of these intersections is concentrated in even degrees (see Proposition \ref{prop:homology 2}).
		
		Explicitly, the Schubert cell $I w I/I$  is identified with the unipotent Lie group generated by elements of the form $1+E_{i,j}$ for $1\leq i \leq n$, $i<j$, $w^{-1}_i>w^{-1}_j$, where $E_{i,j}$ is the matrix with $t^{\lceil \frac{j}{n}\rceil-1}$ at position $i,j\mod n$ and other entries zero. Let us parametrize the 
		points on $I w I/I$ as follows:
		\[
		v(\un):=\left(1 + \sum_{i,j} u_{i,j} E_{i,j}\right)^{-1} w I,
		\]
		where $u_{i,j}\in\C$ are coordinates. The condition that the element is preserved by $1+\gamma_k$ 
		results in a system of equations of the form
		\[
		u_{i,j} = \text{polynomial in $u_{i', j'}$ with $j'-i'<j-i$}
		\]
		for each pair $(i,j)$ satisfying $j>i+kn$. Thus to coordinatize the intersection $I w I/I \cap Y_k$ we can simply eliminate the variables $u_{i,j}$ with $j>i+kn$ using the obtained equations.
		
		Next we pick $r$, $1\leq r \leq n$ and analyze the condition that $v(\un)\in Z_{k}^r$. By definition, this means
		\[
		t^{-1} e_r \in \left(1 + \sum_{i,j} u_{i,j} E_{i,j}\right)^{-1} w \cO^n,
		\]
		equivalently
		\[
		\left(1 + \sum_{i,j} u_{i,j} E_{i,j}\right) e_r\in t \sum_{i=1}^n \cO e_{w_i}.
		\]
		The left hand side can be written as follows:
		\[
		e_r + \sum_{i,l} u_{i,r+l n} t^{l} e_i.
		\]
		Since $\cO\subset \cO e_{w_i}$, the terms with $l>0$ can be ignored. We see that the condition $v(\un)\in Z_{k}^r$ is equivalent to the following: for each $1\leq i\leq n$ satisfying $w_i\leq n$ we have $w_i\neq r$ and $u_{w_i,r}=0$. Recall that $u_{i,j}$ makes sense only when $i<j$ and $w^{-1}(i)>w^{-1}(j)$. We then have:
		
		$v(\un)\in Z_{k}^r$ if and only if $w^{-1}_r\leq 0$ and for each $1\leq i<r$ such that $w^{-1}_i\geq 1$ we have $u_{i,r}=0$. Denote the set of such indices $i$ by $A_r$.
		
		Notice that the condition only involves variables $u_{i,r}$ with $r\leq n$, and so it does not depend on the variables we eliminated when passing from $IwI/I$ to $IwI/I\cap Y_k$.
		
		It is now clear that the intersection of $IwI/I\cap Y_k$ with $Z_k^{\sn\, o}$ if not empty is always of the form
		\[
		U=\C^N \times \prod_{r:\,s_r=0,w^{-1}_r\leq 0} \C^{d_r}\setminus\{0\} = \BA\setminus \bigcup_r H_r
		\]
		for $d_r=|A_r|$. Here $\BA$ is an affine space and $H_r$ is a coordinate subspace of codimension $d_r$. The $T$-characters that appear in $\C^{d_r}$ correspond to differences of the form $x_i-x_r$ with $i\in A_r$, $i<r$.
		
		We now have an explicit description
		\[
		H^T_*(U) = \C[x_1,\ldots,x_n] / \left(\prod_{i\in A_r} (x_i-x_r) \;|\;s_r=0,w^{-1}_r\leq 0\right).
		\]
		This can be seen as follows: first identify $H_*^T(\BA)$ with the polynomial ring $\C[x_1,\ldots,x_n]$. Then
		the space $U$ is the complement of a subspace arrangement, and the homology of $H_r$ is the ideal generated by the polynomial $f_r=\prod_{i\in A_r} (x_i-x_r)$. The homology of an arbitrary intersection is the ideal generated by the corresponding product of polynomials. The polynomials $f_r$ clearly form a regular sequence: each subsequent polynomial has as a leading term which is a power of a variable that does not appear in previous polynomials. By Proposition  \ref{prop:regular sequence} we conclude that the homologies of $H_r$ generate a distributive lattice. By Proposition \ref{prop:homology 1} we conclude that $H_*^T(\bigcup_r H_r)$ is the ideal generated by all $f_r$, and from the long exact sequence we see that $H_*^T(U)$ is the quotient.
		
		So we have shown that the space $Z_k$ is paved by cells whose equivariant homology is supported in even degrees. Every set of the form $Z_k^\sn\cap V_P$ is a union of cells corresponding to a lower set, so Proposition \ref{prop:homology 2} completes the proof.
	\end{proof}

	\begin{rem}
		Combining Proposition \ref{prop:dist-lat} with Proposition \ref{prop:regular sequence} we obtain another proof of Corollary \ref{cor:yaction}.
	\end{rem}
	
	We now have a geometric
	description of $N_k$:
	\begin{defn}
		\label{def:uk}
		Let $U_{k}=\ypk-Z_k^1\cup \cdots \cup Z^n_{k}$, which
		is open in $Z_k$
	\end{defn}
	\begin{cor}
		\label{cor:nkmod}
		We have that
		\begin{equation}
			\label{eq:umod}
			H_*^T(U_k)=H_*^T(\ypk)/
			\left(H_*^T(Z^1_k)+\cdots+H_*^T(Z^n_k)\right) = N_k
		\end{equation}
		In particular, $H_i^T(U_k)=0$ for $i$ odd.
	\end{cor}
	\begin{proof}
		We have the long exact sequence in equivariant
		Borel-Moore homology:
		\begin{equation}
			\label{eq:mvu}
			\cdots \rightarrow H^T_i(Z^1_k\cup \cdots \cup Z^n_k)
			\rightarrow H_i^T(\ypk)\rightarrow
			H^T_i(U_k)\rightarrow \cdots 
		\end{equation}
		Now apply Proposition \ref{prop:dist-lat} to expand the homology of the union
		as a sum, and also to see that the first map is 
		injective. 
	\end{proof}
	\begin{rem}
		It follows by the previous corollaries that
		\[H_*^T(U_k)=H_*^T(\ypk)\otimes_{\C[\yn]} \C\]
		where $\C$ is the $\C[\yn]$ module on which each 
		$y_i$ acts by zero.
	\end{rem}
	\begin{rem}
		Note that despite the vanishing of odd equivariant homology, 
		$U_k$ is not equivariantly formal, and in fact has odd
		nonequivariant homology. Indeed, equivariant formality would
		imply that $N_k\cong H^T_*(U_k)$ 
		is free over $\C[\xn]$, which it is not.
	\end{rem}
	
	\begin{rem}\label{rem:EM}
		Since the space $U_k$ is paved by spaces whose equivariant Borel-Moore homology is pure, the Hodge structure on $H_*^T(U_k)$ is also pure. Thus by \cite{franz2005weights} the Eilenberg-Moore spectral sequence 
		$\Torc{i}{S}{H^T_{-j}(U_k)}
		\Rightarrow H_{-j-i}^*(U_k)$ degenerates.
	\end{rem}
	
	\begin{rem}\label{rem:Leray}
		Since $H_*^T(Z_k^i)$ generate a distributive lattice, the following complex is a free resolution of $H_*^T(U_k)$ over $S$:
		\[
		\cdots\to \bigoplus_{i<j} H_*^T(Z_k^i\cap Z_k^j) \to \bigoplus_i H_*^T(Z_k^i) \to H_*^T(Z_k).
		\]
		Therefore $\Tor$ is computed by the complex
		\begin{equation}\label{eq:tor complex}
			\cdots\to \bigoplus_{i<j} H_*(Z_k^i\cap Z_k^j) \to \bigoplus_i H_*(Z_k^i) \to H_*(Z_k),
		\end{equation}
		since the intersections of $Z_k^i$ are equivariantly formal (Proposition \ref{prop:zpsinj} (i)). The complex \eqref{eq:tor complex} can now be viewed as the $E_1$ page of the spectral sequence converging to $H_*(U_k)$. Since the Borel-Moore homologies of the intersections of $Z_k^i$ are pure, this spectral sequence degenerates at the $E_2$ page.
	\end{rem}
	
	Using either Remark \ref{rem:EM} or Remark \ref{rem:Leray} we conclude with the following description of the weight filtration on $H_*(U_k)$:
	\begin{cor}\label{cor:weight U_k}
		We have
		\[
		\Gr_{i-*}^W H_{*}(U_k) = \Torc{i}{S}{H^T_*(U_k)}.
		\]
		where $\Gr_{i}^W$ denotes the associated graded for the weight filtration, i.e. $\Torc{0}{S}{H^T_*(U_k)}$ equals to the pure part of the homology, $\Torc{1}{S}{H^T_*(U_k)}$ equals to the part one degree off from pure and so on.
	\end{cor}
	
	\subsection{Open Hessenberg varieties}
	We now study the intersection of the Hessenberg
	paving of Section \ref{sec:hesspav} with $U_k$.
	We have a subvariety
	$Z^{\sn}_{\pi,l}=Z^\sn_{n,l} \cap \rsh_\pi$,
	where
	\begin{equation}
		Z^\sn_{n,l}=\left\{V_1\subset\cdots \subset V_n:
		s_i=1\Rightarrow e_i \in V_{n-l}\right\}
	\end{equation}
	and the intersection is taken in the usual
	flag variety $\Fl_n$, with a similar definition for
	$Z^i_{\pi,l}$ where $\sn$ has a one in position
	$i$ as above.
	We consider the complementary variety to
	the $Z^i_{\pi,l}$:
	\begin{defn}
		\label{def:upil}
		Let $\pi$ be a Dyck path, and $0\leq l \leq n$. 
		Let $U_{\pi,l}=U_{n,l} \cap \rsh_\pi$, where
		\begin{equation}
			\label{eq:urshdef}
			U_{n,l}\cong\left\{V_1\subset \cdots \subset V_n:
			\mbox{$e_i \notin V_{n-l}$ for any $i$}\right\}
		\end{equation}
		which is the complement of the $Z^i_{\pi,l}$ in $\rsh_\pi$.
	\end{defn}
	\begin{rem}
		We may also replace the definition of $U_{n,l}$ by
		taking orthogonal complements in $\Fl_n$:
		\[\left\{V_1\subset \cdots \subset V_n:
		\mbox{$V_l \not \subset H_i$ for any $i$}\right\},\]
		where $H_i$ is the \emph{axis hyperplane} 
		perpendicular to $e_i$, and the isomorphism
		is induced by taking orthogonal complements.
		This may be more natural since it generalizes
		the construction of hyperplane complements in 
		$\mathbb{CP}^n$. In this case the cohomology ring
		$H_T^*(U_{l})$ generalizes the (equivariant)
		Orlik-Solomon algebra.
	\end{rem}
	
	Let $P=Q \cup S_n w$ as in Section 
	\ref{sec:hesspav}, but now supposing that $P,Q\subset W_+$,
	and let $\pi=\pi_k(S_nw)$. 
	\begin{prop}
		\label{eq:uhessprop}
		We have that $U_k \cap (V_P-V_Q)$ is the restriction
		of $E_{w} \cap Y_k$ as an affine bundle over $\rsh_\pi$ to $U_{\pi,l}$,
		where $l$ is the number of indices $i$ with $w_i\in \{1,...,n\}$.
	\end{prop}
	
	\begin{proof}
		We have that 
		$U_{\pi,l}=\rsh_\pi-Z^1_{\pi,l}\cup\cdots \cup Z^n_{\pi,l}$, 
		so it suffices to check that $Z^\sn_k\cap V_P-Z^\sn_k\cap V_Q$ is the restriction of
		$E_{w} \cap Y_k$ to $Z^{\sn}_\pi$ for any $\sn$.
		It suffices to consider
		$\sn=(1^m0^{n-m})$, in which $Z^{\sn}_{\pi,l}$ is a union
		of Hessenberg-Schubert cells. This now follows from
		the description of the Schubert cells in Proposition \ref{prop:asptohess}.
		
	\end{proof}
	
	One can ask if all possible pairs $(\pi,l)$ can appear in
	the above proposition. It is not hard to 
	see that the number of indices
	with $w_i \in \{1,...,n\}$, which is an invariant of
	the coset $S_n w$, is at most the number of
	trailing East steps in $\pi_k(S_n w)$. In terms of
	labels $[\mn,\an]$, this number is the same as the
	number of zeroes in $\mn$, and will be denoted
	$z(\mn)=z(S_nw)$ for $w=\paff(\mn,\an)$. Thus, the
	data of possible $(\pi,l)$ is the same as that of a 
	partial Dyck path, discussed in Section \ref{sec:affperm}.
	We will write $\pi'=(\pi,l)$ for a partial Dyck path
	and write $U_{\pi'}=U_{\pi,l}$. We will also use the
	notation $\pi'_k(S_n w)=(\pi_k(S_n w),z(S_n w))$, and similarly
	for $\pi'_k(\mn,\an)$.

	We have an explicit description of the equivariant
	cohomology $H^*_T(U_{\pi,l})$, which agrees with Borel-Moore homology since $U_{\pi,l}$ is smooth. 
	Let $M_\pi=H^*_T(\rsh_\pi)$,
	which is determined by the GKM relations. Let
	\[A^{\sn}_{n,l}=\left\{\sigma\in S_n:
	s_i=1\Rightarrow \sigma^{-1}_i\leq n-l\right\}\]
	be the subset of fixed points of $Z^\sn_{\pi,l}$,
	and similarly for $A^i_{n,l}$ and $Z^i_{n,l}$.
	We define
	\begin{equation}
		\label{eq:upil}
		N_{\pi,l}=M_\pi/\left( F_{A^1_{n,l}} M_{\pi}+\cdots+F_{A^n_{n,l}} M_{\pi}\right),
	\end{equation}
	where $F_A$ denotes the space of elements supported at the fixed points in $A$.
	\begin{prop}
		\label{prop:upil homology}
		We have that $H^*_T(U_{\pi'})\cong N_{\pi'}$.
	\end{prop}
	
	\begin{proof}
		Since the varieties are $S_n$ rotations
		of Schubert varieties, we have that
		$F_{A_{i,l}} M_{\pi}$ is the image of $H_*^T(Z^i_{\pi,l})$
		under Poincar\'{e} duality. Then we proceed in the
		same way as Corollary \ref{cor:nkmod}, noting
		that $Z^i_{\pi,l}$ also satisfy 
		the lattice property of Proposition \ref{prop:dist-lat}.
	\end{proof}
	
	Remarks \ref{rem:EM} and \ref{rem:Leray} can be repeated for the spaces $U_{\pi'}$ and we obtain the following analogue of Corollary \ref{cor:weight U_k}, where we can replace Borel-Moore homology with cohomology since $U_{\pi'}$ is smooth:
	\begin{cor}\label{cor:weight U pi}
		We have
		\[\Gr_{i+*}^W H^*(U_{\pi'}) = \Torc{i}{S}{H^*_T(U_{\pi'})},\]
		i.e. $\Tor_0$ equals to the pure part of the cohomology, $\Tor_1$ equals to the part one degree off from pure and so on.
	\end{cor}
	
	\subsection{Conjectures about Tor groups}
	
	\label{sec:torconj}
	
	We now present some conjectures that
	would categorify the nabla positivity conjecture.
	
	\begin{defn}
		Let $N$ be a module over $S=\C[\xn]$ with a left action of $S_n$,
		which intertwines the action on $S$.
		We will say that $N$ ``satisfies the Tor property'' if the multiplicity of
		the irreducible representation $\chi_\lambda$ in
		\[S_n\acts \Torc{i}{S}{N}\]
		is zero unless $i=\iota(\lambda)$, the number of boxes in $\lambda$ below the main diagonal (see Conjecture \ref{conj:nabpos}).
	\end{defn}

	We now connect these modules to the nabla positivity conjecture, 
	which is 
	Conjecture \ref{conj:nabpos}.
	\begin{conjecture}
		\label{conj:tor}
		Both $H^T_*(U_k)$ and $H_*^T(U_{\pi'})$ satisfy the 
		Tor property as modules over $S$.
	\end{conjecture}
	Note that we do not conjecture this property for $U_{\pi,l}$
	when $l$ is greater than the number of trailing East steps,
	in which case the conjecture does not appear to hold.
	\begin{thm}
		\label{thm:torimpliesnp}
		The Tor property for $H^T_*(U_k)$ implies
		nabla positivity, Conjecture \ref{conj:nabpos}.
		The Tor property for $H^T_*(U_{\pi'})$ implies the
		Tor property for $H^T_*(U_k)$. 
	\end{thm}
	\begin{proof}

		For the first statement, we apply the plethystic substitution 
		$Y\mapsto Y(1-q)$ to both sides
		of $\eqref{eq:frobquot}$ to get
		\[\sum_{i\geq 0} (-1)^i\frob_{Y,X} \Torc{i}{S}{N_k}=
		\omega_X \nabla^k e_n[XY]=\]
		\begin{equation}
			\label{eq:torimpliesnp}
			\sum_{\lambda} \omega_X \nabla^k s_{\lambda'}(X) s_{\lambda}(Y)
			=\sum_{\lambda,\mu} 
			c_{\lambda',\mu'}(q,t) s_\mu(X)s_\lambda(Y).
		\end{equation}
		It follows that $c_{\lambda',\mu}(q,t)$ is signed positive,
		since only $i=\iota(\lambda)$ contributes in \eqref{eq:torimpliesnp}.
		
		For the second statement, let $P=Q\cup S_n w\subset \awg_+$ be lower sets as above. 
		Then since \eqref{eq:mvu} is short exact
		by Corollary \ref{cor:nkmod},
		we have the long exact sequence
		\[\cdots \rightarrow 
		\Torc{i}{S}{H_*^T(U_k\cap V_Q)}
		\rightarrow\Torc{i}{S}{H_*^T(U_k \cap V_P)}\rightarrow
		\Torc{i}{S}{H_*^T(U_{\pi,l})}\rightarrow \cdots \]
		Where $(\pi,l)=\pi'_k(S_n w)$.
		The vanishing then follows by induction on $P$.
		
	\end{proof}
	
	\subsection{A formula for the Frobenius character}
	
	We now compute the Frobenius character $\frob_{-Y} N_{\pi'}$.
	Given a label $\bn$ and a composition $\tn \in \{0,1\}^n$, 
	define a new composition by
	\begin{equation}
		\label{eq:bntntobn}
		\left(\bn\cdot \tn \right)_i=
		(-1)^{t_i}(|b_i|+t_iN),
	\end{equation}
	where $N=\max(|b_1|,...,|b_n|)$. Here the meaning of the signs refer to the
	super variables described in Section \ref{sec:super}.
	\begin{defn}
		We have a symmetric function
		\begin{equation}
			\label{eq:chidef}
			\chi_{\dyckpath,l}[Y;q]=
			\sum_{\substack{\bn,\tn,i>n-l\Rightarrow t_i=0}} 
			(-1)^{|\tn|}
			q^{\inv_{\dyckpath}(\bn\cdot \tn)}Y_{\bn}.
		\end{equation}
	\end{defn}
	\begin{prop}
		\label{prop:nabchi}
		We have 
		\[\frob_{Y} H^*_T(U_{\pi,l})=
		\frac{1}{(1-q)^n}\chi_{\pi,l}[Y;q].\]
	\end{prop}
	
	\begin{proof}
		
		Proceeding the same way as Corollary \ref{cor:yaction},
		we see that there is an exact sequence
		\begin{equation}
			\label{eq:frobchires}
			\nonumber
			0\rightarrow E_{n-l}\rightarrow \cdots \rightarrow
			E_0 \rightarrow H_*^T(U_{\pi,l})
			\rightarrow 0,\quad E_m=\bigoplus_{|\sn|=m} H_*^T(Z^\sn_{\pi,l})
		\end{equation}
		of $S$-modules,
		where the maps are the same as the differentials
		in \c{C}ech cohomology using the maps induced from the
		inclusion $H_*^T(Z^\sn_{\pi,l})$. 
		The maps are equivariant 
		once the $S_n$-action is twisted by the sign
		representation on the $S_m$-factor, and we have
		\[E_m\cong
		\Ind^{S_n}_{S_m\times S_{n-m}} H_*^T(Z^\sn_{\pi,l})\otimes
		(\sgn_m\boxtimes \triv_{n-m})\]
		when $\sn=(1^m0^{n-m})$ ones,
		noticing that in this case $Z^{\sn}_{\pi,l}$ is preserved
		by $S_m\times S_{n-m}$, as is the fixed point set $A^{\sn}_{n,l}$.
		
		We must therefore compute the Frobenius character of the
		resolution. We have
		\begin{equation}
			\label{eq:ind term}
			\left((\sgn_m\boxtimes \triv_{n-m})H^T_*(Z^\sn_{\pi,l})\right)^{S_{\alpha'}\times S_{\alpha''}}=\frac{q^{-\#D(\pi)}}{(1-q)^n}
			\sum_{\sigma \in S_{\alpha}\times S_{\alpha'}\backslash A^{\sn}_{n,l}}
			q^{\inv_\pi(\sigma)}
		\end{equation}
		where $\alpha,\alpha'$ are compositions of $m,n-m$ respectively,
		the (anti) invariants are with respect to the dot action,
		and $\sigma$ is identified with the minimal representative.
		This follows since $Z^{\sn}_{\pi,l}$ is a union of Hessenberg Schubert varieties and so the character can be computed using the corresponding subset of summands of \eqref{eq:shareshianwachs}.
		As a quasi-symmetric function this is
		\[\frob_{Y_1,Y_0} H^T_*(Z^{\sn}_{\pi,l})=
		\frac{q^{-\#D(\pi)}}{(1-q)^n} \sum_{\bn,\tn,i>n-l\Rightarrow t_i=0} q^{\inv_\pi(\bn)-\#D(\pi)} Y_{\tn,\bn}\]
		where $Y_{\tn,\bn}$ is the product of $Y_{t_i,b_i}$, which is
		symmetric in both sets $Y_1,Y_0$. But in terms of 
		Frobenius characters, 
		induction from $S_m\times S_{n-m}$ combined with the
		sign twist corresponds to evaluating $Y_1=-Y,Y_0=Y$, 
		and the result follows.
		
	\end{proof}
	
	\begin{cor}
		\label{cor:nabchi}
		We have that
		\begin{equation}
			\label{eq:nabchi}
			\nabla^k e_n\left[\frac{XY}{1-q}\right]=
			\sum_{[\mn,\an]} \frac{t^{|\mn|} q^{\dinv_k(\mn,\an)}}{(1-q)^n\aut_q(\mn,\an)} X_\an
			\chi_{\pi'_k(\mn,\an)}[Y;q].
		\end{equation}
	\end{cor}
	We list some basic properties of $\chi_{\partialdyckpath}$
	and give an example.
	\begin{prop}\label{prop:vanishing1}
		We have $\chi_{\dyckpath,l}=0$ for $l=0$.
	\end{prop}
	\begin{proof}
		In this case there is no restriction on the values
		of $t_i$, and we find that $\chi_{\dyckpath,0}[Y;q]=\xi_{\dyckpath}[Y-Y;q]=0$. Obviously, in this case $U_{\pi,l}$ is empty.
	\end{proof}
	
	\begin{prop}\label{prop:vanishing2}
		We have $\chi_{\dyckpath,l}=0$ if 
		there exists $1\leq i<n$ such that 
		$1,\ldots,i$ do not attack $i+1,\ldots,n$. In
		other words, $\dyckpath$ has ``touch points,''
		where the path contacts the diagonal.
	\end{prop}
	\begin{proof}
		Note that $n-l+1$ attacks $n$, so $i\leq n-l$. It
		follows that $\chi_{\dyckpath,l}$ factors as
		$\chi_{\dyckpath,l} = \chi_{\dyckpath_1,0}\chi_{\dyckpath_2,l},$
		where $\dyckpath_1$ the beginning part of $\dyckpath$ of length $i$ and $\dyckpath_2$ are the remaining steps. 
		We have $\chi_{\dyckpath_1,0}=0$ by Proposition \ref{prop:vanishing1}, hence $\chi_{\dyckpath,l}=0$. Note that in this case points of $\rsh_{\dyckpath}$ satisfy $\gamma V_i=V_i$ and therefore $V_i\subset V_{n-l}$ must contain a basis vector, so 
		$U_{\pi,l}$ is empty.
	\end{proof}
	
	It is not hard to prove the following:
	\begin{prop}
		\label{prop:maj}
		Let $(\mn,\an)$ be sorted, and let 
		$\mn'=\mn_{\shuff(\an)^{-1}}$,
		$(\dyckpath,l)=\partialdyckpath(\mn,\an).$
		Then the following are equivalent:
		\begin{enumerate}
			\item There exists $\tau\in S_n$ such that
			$\mn'$ is the
			exponent in the Garsia-Stanton descent polynomial
			$g_{\tau}(\yn)=\yn^{\mn'}$, whose degree
			is $\maj(\tau)$.
			\item We have $l>0$ and $\dyckpath$ has no touch points.
		\end{enumerate}
	\end{prop}
	Combining Propositions \ref{prop:vanishing1}, \ref{prop:vanishing2}, and \ref{prop:maj}, we have:
	\begin{cor}
		\label{cor:vanishing}
		The sum in \eqref{eq:nabchi}
		has at most $n!$ nonzero terms.
	\end{cor}
	We can now give an example of one of the sums
	from Corollary \ref{cor:nabchi}.
	\begin{example}
		Using Proposition \ref{prop:nabchi}, we evaluate
		\[(1-q)^3\nabla_X e_3\left[\frac{XY}{1-q}\right]=
		M_{(3)}{\frac {1}{ \left( 1+q \right)  \left(1+q+q^2\right) }}\chi_{111}+\]
		\[M_{(2,1)}\left({\frac {t}{1+q}}\chi_{{1110}}+{\frac {1}{1+q}}\chi_{111}+{t}^{2}\chi_{11010}\right)+\]
		\[M_{(1,2)}\left(t \chi_{{1101}}+{\frac {1}{1+q}}\chi_{111}+{\frac {{t}^{2
			}}{1+q}}\chi_{11100}\right)+\]
		\[M_{(1,1,1)}\left(\left( t^2+t^3 \right) \chi_{{11010}}+t\chi_{{1110}}+t\chi_{{1101}}
		+\chi_{{111}}+t^2 \chi_{{11100}}\right).\]
		Here $M_\alpha=M_\alpha(X)$ is the quasi-symmetric
		monomial in the $X$ variables, and have 
		eliminated all terms with no contribution
		using Corollary \ref{cor:vanishing}.
		
		Now writing $\chi_{\partialdyckpath}'=
		(1-q)^{-n}\chi_{\partialdyckpath}[Y(1-q);q]$,
		we get
		\[\chi'_{1101}= \left( 1+q \right) s_{{3}}-{q}^{2}s_{{2,1}}\]
		\[\chi'_{11010}=
		{q}^{2}s_{{111}}-qs_{{2,1}}+s_{{3}}\]
		\[\chi'_{11100}=
		\left( 1+q\right) s_{{3}}-{q} \left( 1+q \right) s_{{2,1}}+{q}^{2} \left( 1+q \right) s_{{1,1,1}}\]
		\[\chi'_{111}=
		\left( 1+q \right)  \left(1+q+q^2\right) s_{{3}}\]
		\[\chi'_{1110}=
		\left( 1+q \right)^{2}s_{{3}}-{q}^{2} \left( 1+q \right) s_{{2,1}}.\]
		Notice that each factor satisfies the signed
		positivity property, and is divisible
		by the automorphism factors in the above equation.
		Moreover, the factors in front of $M_{(1,2)}$
		and $M_{(2,1)}$ are equal, so we do in fact
		end up with a symmetric function.
	\end{example}
	
	\begin{example} Conjecture \ref{conj:tor} can be used to compute
		the cohomology groups of the $U_{\pi'}$ and $U_k$, and we give
		an example for $U_{\pi'}$: we calculate
		\[\chi_{11100110}=- \left( 1+q \right) ^{2} \left( q-1 \right) ^{3}s_{{5}}-q \left( 1+q
		\right) ^{2} \left( q-1 \right) ^{3}s_{{4,1}}.\]
		Recall that Borel-Moore homology agrees with cohomology for
		smooth spaces, so we have $H^*_T(U_{\pi'})=H_*^T(U_{\pi'})$.
		From this we compute the signed Frobenius character of 
		$H_*(U_{\pi'})$, by applying the substitution
		\[F[Y]\mapsto (1-q)^{-n}F[(1-q)Y]
		\big|_{s_\lambda=x^{-\iota(\lambda)}s_\lambda,q=x^2},\]
		where $x$ is the generating variable for degree. We obtain
		\[ \left( {x}^{2}+1 \right) ^{3}s_{{5}}-{x}^{3} \left( {x}^{2}+1 \right) 
		^{2}s_{{4,1}}-{x}^{3} \left( {x}^{2}+1 \right) ^{2}s_{{3,2}}+\]
		\[{x}^{4}
		\left( {x}^{2}+1 \right) ^{2}s_{{3,1,1}}+
		{x}^{4} \left( {x}^{2}+1
		\right) ^{2}s_{{2,2,1}}- \left( {x}^{2}+1 \right) ^{2}{x}^{5}s_{{2,1,1
				,1}},\]
		noticing that the signs respect the parity of the degree in $x$.
	\end{example}
	
	\appendix
	\section{Distributive lattices}
	We collect some useful facts and definitions about collections of subspaces of a fixed vector space here.
	\begin{defn}
		Let $\CL$ be a collection of subspaces of a fixed vector space $\CV$. We say $\CL$ is a \emph{lattice} if for any subspaces $A,B\in\CL$ we have $A\cap B, A+B \in \CL$. If $\CL$ is not necessarily a lattice, we can always form a lattice by taking all possible combinations of the operations $\cap$ and $+$ and applying them to all collections of elements of $\CL$. A lattice thus obtained is called the \emph{lattice generated by} $\CL$. A lattice $\CL$ is called \emph{distributive} if for any $A,B,C$ we have
		\[
		(A+B)\cap C = (A\cap C) + (B\cap C).
		\]
	\end{defn}
	
	The condition above can be replaced by the condition
	\[
	(A\cap B)+C = (A+C)\cap (B+C).
	\]
	
	\begin{defn}
		Given a vector space $\CV$ with a choice of a basis $\{e_i\}_{i\in I}$, a subspace $A\subset \CV$ is called a \emph{coordinate subspace} if there exists a subset $J\subset I$ such that $A$ is the span of $\{e_i\}_{i\in J}$.
	\end{defn}
	
	It is clear that all coordinate subspaces of a given vector space with a given basis form a distributive lattice. Conversely, we have
	\begin{prop}
		Suppose $\CL$ is a finite distributive lattice of subspaces of $\CV$. There exists a basis of $\CV$ such that all elements of $\CL$ are coordinate subspaces.
	\end{prop}
	\begin{proof}
		Replacing $\CV$ by $\sum_{A\in \CL} A$ or by $\CL/\bigcap_{A\in\CL} A$ if necessary we may assume that $\{0\},\CV$ are both in $\CL$. Now we induct on the number of elements of $\CL$. The base case $|\CL|=2$ is obvious. For the induction step, let $A\in \CL$ be any maximal element of $\CL$ not equal to $\CV$, and let $B$ be any minimal element of $\CL$ not contained in $A$. Note that we have $A+B=\CV$ by the maximality of $A$. Elements of $\CL$ contained in $A$ form a distributive lattice. Applying the induction assumption, choose any basis of $A$ such that all elements of $\CL$ contained in $A$ are coordinate subspaces, so in particular $A\cap B$ is a coordinate subspace. Complete the basis of $A\cap B$ to a basis of $B$, and we obtain a basis of $\CV$ such that any element of $\CL$ contained in $A$ is a coordinate subspace, and $B$ is a coordinate subspace. Now let $C \in \CL$ be any element and let us prove that $C$ is a coordinate subspace. If $C\subset A$ we are done. Otherwise, apply 
		distributivity to write
		\[
		C = C\cap (A+B) = (C\cap A) + (C\cap B).
		\]
		Here $C\cap B\subset B$ is such that $C\cap B\nsubset A$. By the minimality of $B$ we have $C\cap B=B$, so we have $C=(C\cap A) + B$. Since both $C\cap A$ and $B$ are coordinate subspaces, $C$ is as well.
	\end{proof}
	\begin{rem}
		The statement is not true in general for infinite lattices. For a counterexample, let $\CV=\C[t]$ and let $H_x$ denote the subspace of polynomials vanishing at $x\in\C$. Suppose there exists a basis of $\CV$ so that all $H_x$ 
		are coordinate subspaces, and let $e$ be any basis vector which is not in $H_0$. Then $e$ must belong to $H_x$ for all $x\neq 0$ and therefore $e=0$, which is a
		contradiction.
	\end{rem}
	
	\subsection{Regular sequences}
	\begin{defn}
		A sequence of commuting endomorphisms $f_1,\ldots,f_n$ of a vector space $M$ is called \emph{regular} if for any $k=1,\ldots,n$ and any $m\in M$ satisfying $f_k m \in (f_1,\ldots,f_{k-1}) M$ we have $m_k\in (f_1,\ldots,f_{k-1}) M$.
	\end{defn}
	\begin{prop}\label{prop:regular sequence}
		Suppose $M$ is a non-negatively graded vector space, and suppose $f_1,\ldots,f_n$ are commuting endomorphisms of $M$ of positive degree. Then the following conditions are equivalent:
		\begin{enumerate}
			\item The sequence $f_1,\ldots,f_n$ is regular.
			\item $M$ is a free module over the polynomial ring $\C[f_1,\ldots,f_n]$.
			\item $f_i$ is injective for all $i$, $f_i M \cap f_j M = f_i f_j M$ for all $i\neq j$ and the subspaces $f_i M\subset M$ generate a distributive lattice.
		\end{enumerate}
	\end{prop}
	\begin{proof}
		The implication $(i)\Rightarrow(ii)$ is standard. To show that $(ii)\Rightarrow(iii)$ let $(m_i)_{i\in I}$ be a basis of $M$ as a module over the polynomial ring. Then $f_1^{a_1} \cdots f_n^{a_n} m_i$ form a basis of $M$ as a vector space. Clearly, $f_i M$ are coordinate subspaces for this basis. Hence they generate a distributive lattice.
		
		Now let us show that $(iii)\Rightarrow (i)$. Suppose $f_k m \in (f_1,\ldots,f_{k-1}) M$. This implies
		\[
		f_k m \in f_k M \cap (f_1 M + \cdots + f_{k-1} M) = (f_k M \cap f_1 M) + \cdots + (f_k M \cap f_{k-1} M),
		\]
		and using $f_k M \cap f_i M = f_k f_i M$ and the injectivity of $f_k$ we obtain $m\in  f_1 M + \cdots + f_{k-1} M$.
	\end{proof}
	
	\subsection{Homology}
	For simplicity of notation, we formulate results in this section for ordinary Borel-Moore homology, but all statements clearly remain valid for equivariant Borel-Moore homology.
	
	\begin{defn}
		Suppose $X$ is a topological space. A collection of closed subsets $\CZ$ is called a \emph{lattice} if for any $Z_1, Z_2\in \CZ$ we have $Z_1\cap Z_2, Z_1\cup Z_2 \in \CZ$. A lattice of subsets is called \emph{nice} if for any $Z\in \CZ$ the map on Borel-Moore homology $H_*(Z)\to H_*(X)$ is injective and for any $Z_1, Z_2\in\CZ$ we have $H_*(Z_1\cap Z_2)=H_*(Z_1)\cap H_*(Z_2)$, $H_*(Z_1\cup Z_2)=H_*(Z_1)+H_*(Z_2)$.
	\end{defn}
	
	\begin{prop}\label{prop:homology 1}
		Suppose $\CZ$ is a collection of closed subsets of $X$, and suppose for any $Z\in \CZ$ the map on Borel-Moore homology $H_*(Z)\to H_*(X)$ is injective, and suppose for any tuple $Z_1,\ldots, Z_m$ the map on Borel-Moore homology
		\[
		H_*(Z_1\cap\cdots\cap Z_m) \to H_*(Z_1)\cap\cdots \cap H_*(Z_m)
		\]
		is an isomorphism. Then the following conditions are equivalent:
		\begin{enumerate}
			\item The lattice generated by $\CZ$ is nice.
			\item The subspaces $H_*(Z)\subset H_*(X)$ for $Z\in\CZ$ generate a distributive lattice.
		\end{enumerate}
	\end{prop}
	\begin{proof}
		$(i)\Rightarrow (ii)$ is evident because if the lattice generated by $\CZ$ is nice, then the operations on vector spaces match with the operations on subsets, while on subsets the distributive law $(Z_1\cup Z_2)\cap Z_3 = (Z_1\cap Z_3)\cup (Z_2\cap Z_3)$ is automatic.
		
		Now we show $(ii)\Rightarrow(i)$. Without loss of generality we may assume that $\CZ$ is closed under intersections. Then any element in the lattice generated by $\CZ$ can be written as a union of elements of $\CZ$. Let us show that $H_*(Z_1\cup\cdots\cup Z_m)=H_*(Z_1)+\cdots H_*(Z_m)$ for any $Z_1,\ldots,Z_m$ by induction on $m$. Let $A=Z_1\cup\cdots\cup Z_{m-1}$. We have a long exact sequence
		\begin{equation}\label{eq:long exact sequence}
			\cdots \to H_*(A\cap Z_m) \to H_*(A)\oplus H_*(Z_m) \to H_*(A\cup Z_m)\to \cdots
		\end{equation}
		By the induction assumption, the assumptions on $\CZ$, 
		and the distributivity assumption we have
		\[
		H_*(A\cap Z_{m-1}) = H_*((Z_1\cap Z_m)\cup \cdots \cup (Z_{m-1}\cap Z_m))\]
		\[= H_*(Z_1\cap Z_m)+\cdots+H_*(Z_{m-1}\cap Z_m)
		\]
		\[
		=(H_*(Z_1)\cap H_*(Z_m))+\cdots+(H_*(Z_{m-1})\cap H_*(Z_m)) 
		\]
		\[
		= (H_*(Z_1)+\cdots+H_*(Z_{m-1})) \cap H_*(Z_m) = H_*(A) \cap H_*(Z_m).
		\]
		Thus in particular the second arrow in \eqref{eq:long exact sequence} is injective and the long exact sequence splits into short exact sequences. This implies that $H_*(A\cup Z_m)$ is the quotient $H_*(A)\oplus H_*(Z_m)/H_*(A)\cap H_*(Z_m)$, which is isomorphic to the sum $H_*(A)+ H_*(Z_m)$, so the induction step is proved.
		
		Now suppose $A=\bigcup_{i} Z_i$, $B=\bigcup_i Z_i'$ are arbitrary elements of the lattice generated by $\CZ$. We have 
		\[
		H_*(A) = \sum_i H_*(Z_i), \quad H_*(B) = \sum_i H_*(Z_i'),
		\]
		and therefore
		\[
		H_*(A) + H_*(B) = \sum_i H_*(Z_i) + \sum_i H_*(Z_i') = H_*(A\cup B),
		\]
		\[
		H_*(A) \cap H_*(B) = \sum_{i,j} H_*(Z_i)\cap H_*(Z_j') = H_*\left(\bigcup_{i,j} Z_i\cap Z_j'\right)=H_*(A\cap B).
		\]
	\end{proof}
	
	A useful tool for constructing nice lattices are affine pavings.
	\begin{prop}\label{prop:homology 2}
		Suppose $X$ is paved by finitely many sets $Z_\alpha$ indexed by $\alpha\in\Lambda$ where $\Lambda$ is a poset such that for any lower set $A\subset\Lambda$ the union $Z_A:=\bigcup_{\alpha\in A} Z_\alpha$ is closed. Suppose for any $\alpha\in \Lambda$ the odd homologies $H_{2i+1}(Z_\alpha)$ vanish. Then the subsets of the form $Z_A$ form a nice lattice of sets.
	\end{prop}
	\begin{proof}
		Applying long exact sequences and the vanishing of odd homology we see that the odd homologies of $Z_A$ vanish for any finite lower set $A$ and for any $A\subset B$ the map $H_*(Z_A)\to H_*(Z_B)$ is injective. Now consider arbitrary lower sets $A$, $B$. We have a short exact sequence
		\[
		0\to H_*(Z_A\cap Z_B) \to H_*(Z_A)\oplus H_*(Z_B) \to H_*(Z_A\cup Z_B)\to 0.
		\]
		Since $H_*(Z_A\cup Z_B)$ embeds into $H_*(X)$ we obtain that $H_*(Z_A\cup Z_B)=H_*(Z_A)+H_*(Z_B)$ and $H_*(Z_A\cap Z_B)=H_*(Z_A)\cap H_*(Z_B)$.
	\end{proof}
	
	\bibliographystyle{amsalpha}
	\bibliography{refs}

\newcommand{\etalchar}[1]{$^{#1}$}
\providecommand{\bysame}{\leavevmode\hbox to3em{\hrulefill}\thinspace}
\providecommand{\MR}{\relax\ifhmode\unskip\space\fi MR }
\providecommand{\MRhref}[2]{%
  \href{http://www.ams.org/mathscinet-getitem?mr=#1}{#2}
}
\providecommand{\href}[2]{#2}
\begin{thebibliography}{BGHT99}

\bibitem[AHM17]{abe2017cohomology}
Hiraku Abe, Tatsuya Horiguchi, and Mikiya Masuda, \emph{The cohomology rings of
  regular semisimple {H}essenberg varieties for $h=(h(1),n,\ldots,n)$}, Journal
  of Combinatorics \textbf{10} (2017), 24.

\bibitem[AL21]{alvarez2021affine}
Pablo~Boixeda Alvarez and Ivan Losev, \emph{Affine {S}pringer fibers, {P}rocesi
  bundles, and {C}herednik algebras}, arXiv preprint arXiv:2104.09543 (2021),
  43.

\bibitem[BC15]{brosnan2018hessenberg}
Patrick Brosnan and Timothy Chow, \emph{Unit interval orders and the dot action
  on the cohomology of regular semisimple {H}essenberg varieties}, Advances in
  Mathematics \textbf{329} (2015), 955--1001.

\bibitem[BGHT99]{bergeron1999identities}
Francois Bergeron, Adriano~M. Garsia, Mark Haiman, and Glenn Tesler,
  \emph{Identities and positivity conjectures for some remarkable operators in
  the theory of symmetric functions}, Methods Appl. Anal. \textbf{6} (1999),
  no.~3, 363--420. \MR{1803316}

\bibitem[BM83]{borho1983partial}
Walter Borho and Robert MacPherson, \emph{Partial resolutions of nilpotent
  varieties}, Analyse et topologie sur les espaces singuliers (II-III) - 6 - 10
  juillet 1981, Ast\'erisque, no. 101-102, Soci\'et\'e math\'ematique de
  France, 1983 (en). \MR{737927}

\bibitem[Bri98]{brion1998equivariant}
Michel Brion, \emph{Equivariant cohomology and equivariant intersection
  theory}, pp.~1--37, Springer Netherlands, Dordrecht, 1998.

\bibitem[Bri07]{bridgeland2007stability}
Tom Bridgeland, \emph{Stability conditions on triangulated categories}, Annals
  of Mathematics \textbf{166} (2007), no.~2, 317--345.

\bibitem[CM18]{carlsson2018proof}
Erik Carlsson and Anton Mellit, \emph{A proof of the shuffle conjecture}, J.
  Amer. Math. Soc. \textbf{31} (2018), no.~3, 661--697. \MR{3787405}

\bibitem[CM20]{carlsson2020combinatorial}
Erik Carlsson and Anton Mellit, \emph{A combinatorial formula for the nabla
  operator}, arXiv preprint arXiv:2012.01627 (2020), 35.

\bibitem[DMPS92]{demari1992hessenberg}
Filippo De~Mari, Claudio Procesi, and M.~Shayman, \emph{Hessenberg varieties},
  Trans. Amer. Math. Soc. \textbf{332} (1992), 529--534.

\bibitem[EG96]{edidin96equivariant}
Dan Edidin and William Graham, \emph{Equivariant intersection theory}, Invent.
  Math \textbf{131} (1996), 595--634.

\bibitem[FW05]{franz2005weights}
Matthias Franz and Andrzej Weber, \emph{Weights in cohomology and the
  {E}ilenberg-{M}oore spectral sequence}, Ann. Inst. Fourier (Grenoble)
  \textbf{55} (2005), no.~2, 673--691. \MR{2147902}

\bibitem[GHZ06]{guillemin2006description}
Victor Guillemin, Tara Holm, and Catalin Zara, \emph{A {GKM} description of the
  equivariant cohomology ring of a homogeneous space}, Journal of Algebraic
  Combinatorics \textbf{23} (2006), 21--41.

\bibitem[GKM97]{goresky1997koszul}
M.~Goresky, R.~Kottwitz, and R.~MacPherson, \emph{Equivariant cohomology,
  {K}oszul duality, and the localization theorem}, Invent math \textbf{131}
  (1997), 25--83.

\bibitem[GKM03]{goresky2003purity}
M.~Goresky, Robert~E. Kottwitz, and R.~Macpherson, \emph{Purity of equivalued
  affine {S}pringer fibers}, Representation Theory of The American Mathematical
  Society \textbf{10} (2003), 130--146.

\bibitem[GKM04]{goresky2004unramified}
Mark Goresky, Robert Kottwitz, and Robert MacPherson, \emph{Homology of affine
  {S}pringer fibers in the unramified case}, Duke Math. J. \textbf{121} (2004),
  no.~3, 509--561.

\bibitem[Gra01]{graham2001positivity}
William Graham, \emph{{Positivity in equivariant Schubert calculus}}, Duke
  Mathematical Journal \textbf{109} (2001), no.~3, 599 -- 614.

\bibitem[GZ03]{guillemin2003existence}
Victor Guillemin and Catalin Zara, \emph{The existence of generating families
  for the cohomology ring of a graph}, Advances in Mathematics \textbf{174}
  (2003), 115--153.

\bibitem[Hag08]{haglund2008catalan}
James Haglund, \emph{The {$q$},{$t$}-{C}atalan numbers and the space of
  diagonal harmonics}, University Lecture Series, vol.~41, American
  Mathematical Society, Providence, RI, 2008, With an appendix on the
  combinatorics of Macdonald polynomials. \MR{2371044 (2009f:05261)}

\bibitem[Hai01a]{haiman2001hilbert}
Mark Haiman, \emph{Hilbert schemes, polygraphs and the {M}acdonald positivity
  conjecture}, Journal of the American Mathematical Society \textbf{14} (2001),
  no.~4, 941--1006.

\bibitem[Hai01b]{Haiman01vanishingtheorems}
Mark Haiman, \emph{Vanishing theorems and character formulas for the {H}ilbert
  scheme of points in the plane}, Invent. Math \textbf{149} (2001), 371--407.

\bibitem[HHL{\etalchar{+}}05]{haglund2005combinatoriala}
J.~Haglund, M.~Haiman, N.~Loehr, J.~B. Remmel, and A.~Ulyanov, \emph{A
  combinatorial formula for the character of the diagonal coinvariants}, Duke
  Math. J. \textbf{126} (2005), no.~2, 195--232. \MR{2115257}

\bibitem[HMZ12]{haglund2012compositional}
J.~Haglund, J.~Morse, and M.~Zabrocki, \emph{A compositional shuffle conjecture
  specifying touch points of the {D}yck path}, Canad. J. Math. \textbf{64}
  (2012), no.~4, 822--844. \MR{2957232}

\bibitem[Kiv20]{kivinen2020unramified}
Oscar Kivinen, \emph{Unramified affine {S}pringer fibers and isospectral
  {H}ilbert schemes}, Selecta Mathematica \textbf{26} (2020), 42.

\bibitem[KK86]{kostant1986nil}
Bertram Kostant and Shrawan Kumar, \emph{The nil {H}ecke ring and cohomology of
  {$G/P$} for a {K}a\c{c}-{M}oody group {$G$}}, Proceedings of the National
  Academy of Sciences of the United States of America \textbf{83} (1986),
  1543--5.

\bibitem[KT01]{knutson2001puzzles}
Allen Knutson and Terence Tao, \emph{Puzzles and (equivariant) cohomology of
  {G}rassmannians}, Duke Mathematical Journal \textbf{119} (2001), 30.

\bibitem[Kum02]{kumar2002kac}
S.~(Shrawan) Kumar, \emph{Kac-{M}oody groups, their flag varieties, and
  representation theory}, Progress in mathematics ; v. 204, Birkhäuser,
  Boston, 2002 (eng).

\bibitem[LLM{\etalchar{+}}14]{lam2014schur}
Thomas Lam, Luc Lapointe, Jennifer Morse, Anne Schilling, Mark Shimozono, and
  Mike Zabrocki, \emph{$k$-{S}chur functions and affine {S}chubert calculus},
  vol.~33, Springer, 2014.

\bibitem[LS91]{lusztig1991fixed}
G.~Lusztig and J.~M. Smelt, \emph{Fixed point varieties on the space of
  lattices}, Bulletin of The London Mathematical Society \textbf{23} (1991),
  213--218.

\bibitem[LSS10]{lam2010affine}
Thomas Lam, Anne Schilling, and Mark Shimozono, \emph{K-theory schubert
  calculus of the affine {G}rassmannian}, Compositio Mathematica \textbf{146}
  (2010), no.~4, 811–852.

\bibitem[LW08]{loehr2008nested}
Nicholas Loehr and Gregory Warrington, \emph{Nested quantum {D}yck paths and
  del(s(lambda))}, IMRN. International Mathematics Research Notices
  \textbf{2008} (2008), 23.

\bibitem[Mac95]{macdonald1995symmetric}
Ian~G. Macdonald, \emph{Symmetric functions and {H}all polynomials}, second
  ed., Oxford Mathematical Monographs, The Clarendon Press, Oxford University
  Press, New York, 1995, With contributions by A. Zelevinsky, Oxford Science
  Publications. \MR{1354144}

\bibitem[Mel20]{mellit2020poincare}
Anton Mellit, \emph{Poincar\'{e} polynomials of character varieties,
  {M}acdonald polynomials and affine {S}pringer fibers}, Ann. of Math. (2)
  \textbf{192} (2020), no.~1, 165--228. \MR{4125451}

\bibitem[SW12]{shareshian2012chromatic}
John Shareshian and Michelle~L. Wachs, \emph{Chromatic quasisymmetric functions
  and {H}essenberg varieties}, Configuration Spaces (Pisa) (A.~Bjorner,
  F.~Cohen, C.~De~Concini, C.~Procesi, and M.~Salvetti, eds.), Scuola Normale
  Superiore, 2012, pp.~433--460.

\bibitem[SW16]{shareshian2016chromatic}
\bysame, \emph{Chromatic quasisymmetric functions}, Adv. Math. \textbf{295}
  (2016), 497--551. \MR{3488041}

\bibitem[Tym05]{tymoczko2005introduction}
J~S Tymoczko, \emph{{An introduction to equivariant cohomology and homology,
  following Goresky, Kottwitz, and MacPherson}}, Tech. Report math.AG/0503369,
  CERN, Mar 2005.

\bibitem[Tym08]{tymoczko2008permutation}
Julianna~S. Tymoczko, \emph{Permutation representations on {S}chubert
  varieties}, American Journal of Mathematics \textbf{130} (2008), no.~5,
  1171--1194.

\end{thebibliography}
	
\end{document}